\documentclass[a4paper,10pt]{article}

\usepackage{fullpage}

\usepackage[utf8]{inputenc}
\usepackage{amsfonts}
\usepackage{amssymb}
\usepackage{amsmath}
\usepackage{amsthm}
\usepackage{tikz-cd}
\usepackage{url}
\usepackage{setspace}
\usepackage{enumerate}

\usepackage[T1]{fontenc}
\usepackage[all]{xy}

\usepackage{xcolor}
\usepackage{soul}

\theoremstyle{plain}
\newtheorem{defi}{Definition}[section]
\newtheorem{thm}{Theorem}[section]
\newtheorem{prop}{Proposition}[section]
\newtheorem{lem}{Lemma}[section]
\newtheorem{coro}{Corollary}[section]
\newtheorem{rem}{Remark}[section]
\newcommand{\N}{\mathbb{N}}
\newcommand{\Z}{\mathbb{Z}}

\newcommand{\F}{\mathbb{F}}
\newcommand{\pr}{\mathbb{P}^1}

\DeclareMathOperator{\Spm}{Spm}

\DeclareMathOperator{\Aut}{Aut}

\DeclareMathOperator{\Ker}{Ker}

\DeclareMathOperator{\res}{res}

\DeclareMathOperator{\coeff}{coeff}

\def\ms{\medskip}

\def\msn{\medskip\noindent}

\title{On M-O.Ore determinants\footnote{A tribute to 
 the contributions of O. Ore \cite{Or} without forgetting that E.H. Moore defined the said determinant \cite{Mo} }}

\author{Jean Fresnel, Michel Matignon }

\begin{document}

\maketitle
%\vspace*{\stretch{1}}
\begin{flushright}
To  Marco Garuti, 

a friend too soon lost
\end{flushright}
%\vspace*{\stretch{2}}

\begin{abstract}
%\centerline{\bf Résumé}

The existence of certain $\F_q$-spaces of differential forms of the projective line over a field $K$ containing $\F_q$ leads us to prove an identity linking the determinant of the Moore matrix of $n$ indeterminates with the determinant of the Moore matrix of the cofactors of its first row. These same spaces give an interpretation of Elkies pairing in terms of residues of differential forms. This pairing puts in duality the $\F_q$-vector space of the roots of a $\F_q$-linear polynomial and that of the roots of its reversed polynomial.

\end{abstract}

\section{Introduction} 

Marco Garuti was rapporteur for Guillaume Pagot's thesis (\cite {P1, P2}) and the origin of this note is the following remark (\cite {P2} p. 68).

Let $K$ be a field with characteristic $p>0$ and  $W\subset K[X],$ be an $n$ dimensional  $\F_p$-subspace in $K[X]$ whose non zero elements  have the same degree $d$ and let $P$ be a non zero polynomial which is a common multiple  of the polynomials in $W$. Let  $(P_1,P_2,\cdots,P_n)$, be a $\F_p$-basis of $W$ such that for  $1\leq i\leq n$ each differential form  $\omega_i:=\frac{P_i}{P}dX$ is a logarithmic differential, then there is  $\gamma\in K^\star$ such that 
\begin{eqnarray}
\Delta_n(P_1,P_2,\cdots,P_n)=\gamma P^{1+p+p^2+\cdots+p^{n-2}}
\end{eqnarray}\label{pagot}
\noindent where $\Delta_n(P_1,P_2,\cdots,P_n)$
is the Moore determinant of the $n$ polynomials  $P_1,P_2,\cdots,P_n$ (Definition \ref{defi2.1}).

In Proposition \ref{propMAT}, we adapt the method of (\cite{Ma,P2}) where such a $\F_p$-space $W$ for $n\geq 2$ was built, in order to build some $\F_q$-space  of differential forms in  $ \Omega^1_K(K(X))$ 
that we call a $L^q_{\mu+1,n}$-space.

In section 3.2, we give a first application to a construction of Elkies pairing (\cite{E} § 4.35). N. Elkies takes up and extends the results of O. Ore.  For a $\F_q$-linear polynomial 
$P:=c_0X+c_1X^q+\cdots+c_nX^{q^n}$ with $c_0c_n\neq 0$, his pairing induces a duality between the $\F_q$-space of roots of  $P$ and that of its reversed polynomial. In our construction the role of the $\F_q$-vector space of the roots of the reversed polynomial is played by a $L^q_{\mu+1,n}$-space  of differential forms and the pairing is expressed by the residue of these forms evaluated at the roots of the polynomial $P$.

The rest of the note deals with the evaluation of the constant $\gamma$ 
in \eqref{pagot} when we apply it to $L^q_{\mu+1,n}$-subspaces  of $L^q_{\mu+1,n+m}$-spaces constructed in Proposition \ref{propMAT}. 

Thus formula \eqref{pagot} takes the following form (Corollary 
\ref{coro2.3})

\msn {\bf Corollary \ref{coro2.3}.}
{\em Let  $(\underline{Y}):=(Y_1,Y_2,\cdots,Y_n)$ and
$(\underline{X}):=(X_1,X_2,\cdots,X_m)$, be  $n+m$ indeterminates over $\F_q$ with $n\geq 2,\ m\geq 0$ and the convention that $\underline{X}=\emptyset $  and $\Delta_m(\underline{X})=1$ for $m=0$. 
We  write $(\underline{\hat Y_i}):=(Y_1,\cdots,Y_{i-1},Y_{i+1}\cdots,Y_n)$ for $1\leq i\leq n$. Then we have the following equality in $\F_q(\underline Y,\underline X)$ 

$$\frac{\Delta_n(\Delta_{n-1+m}(\underline{\hat Y_1},\underline{X}),\cdots,(-1)^{i+1}\Delta_{n-1+m}(\underline{\hat Y_i},\underline{X}),\cdots,(-1)^{n+1}\Delta_{n-1+m}(\underline{\hat Y_n},\underline{X}))}{
\Delta_m(\underline{X})^{q^{n-1}}\Delta_{n+m}(\underline{Y},\underline{X})^{1+q+\cdots+q^{n-2}}}=$$
$$
\frac{\Delta_n(\Delta_{n-1}(\underline{\hat Y_1}),\cdots,(-1)^{i+1}\Delta_{n-1}(\underline{\hat Y_i}),\cdots,(-1)^{n+1}\Delta_{n-1}(\underline{\hat Y_n}))}{
\Delta_{n}(\underline{Y})^{1+q+\cdots+q^{n-2}}}=:\gamma 
$$
}

Thanks to the work of Ore and Elkies (cf. Proposition \ref{Nprop2.4}) we know that $\gamma\in \F_q^\star$.

We show (Theorem \ref{thm1}) that $\gamma=1$.
We give three proofs. The first one shows it for $m=1$ and by induction on $n$. It is a technical exercise in computing determinants. The second proof is a matrix equality which is in itself original and which translates 
a relation between a generic Moore matrix and the Moore matrix of the cofactors of its first row (see theorem below); a relation analogous to the classical relation between a square matrix and its comatrix. The $m=0$ case of Theorem \ref{thm1}  is immediately deduced.

\msn {\bf Theorem \ref{inverse}.}
{\em Let $Y_1,Y_2,\cdots,Y_n$, $n$ be indeterminates over $\F_q$ and 
 
 \noindent ${\cal{M}}_{n}(\Delta_n(\underline{\hat Y_1}),.,(-1)^{i-1}\Delta_n(\underline{\hat Y_i}),
\cdots,(-1)^{n-1}\Delta_n(\underline{\hat Y_{n}}))$ be the Moore matrix of the cofactors 

\noindent $(\Delta_n(\underline{\hat Y_1}),.,(-1)^{i-1}\Delta_n(\underline{\hat Y_i}),
\cdots,(-1)^{n-1}\Delta_n(\underline{\hat Y_{n}}))$ of the first row of ${\cal{M}}_n(\underline{Y})$ where for $1\leq i\leq n$. We write $(\underline{\hat Y_i}):=(Y_1,\cdots,Y_{i-1},Y_{i+1}\cdots,Y_n)$,
then we have
$$
 {\cal{M}}_{n}(\Delta_n(\underline{\hat Y_1}),\cdots,(-1)^{i-1}\Delta_n(\underline{\hat Y_i}),
\cdots,(-1)^{n-1}\Delta_n(\underline{\hat Y_{n}}))\ ^t{\cal{M}}_n(\underline Y)=
$$
$$
=\left(
\begin{array}{ccccccc}
0&.&.&\cdots&.&0& (-1)^{n-1}\Delta_n(\underline Y)  \\
\Delta_n(\underline Y)&0&.&\cdots&.&0&0\\
\alpha_1&\Delta_n(\underline Y)^q&0&\cdots&.&0&0  \\
\alpha_2&\alpha_1^q&\Delta_n(\underline Y)^{q^2}&\cdots&.&0&0\\
\vdots&\vdots&\vdots&\cdots&.&\vdots&\vdots\\
\alpha_{n-2}&\alpha_{n-3}^q&.&\cdots&\alpha_1^{q^{n-3}}&\Delta_n(\underline Y)^{q^{n-2}}&0

\end{array}
\right)
$$
where $\alpha_k:=
\Delta_n(\underline{\hat Y_1})^{q^{k+1}}Y_1+\cdots+(-1)^{i-1}\Delta_n(\underline{\hat Y_i})^{q^{k+1}}Y_i+\cdots+(-1)^{n-1}\Delta_n(\underline{\hat Y_n})^{q^{k+1}}Y_n $.
}

\medskip
The third proof is a generalization of the above theorem which gives a matrix equality (Theorem \ref{thm2}) from which we deduce  Theorem \ref{thm1} without invoking  Corollary \ref{coro2.3}.

\ms
In section 5, we offer two illustrations of M-O.Ore determinants. In the first one we study the application $(a_1,\cdots,a_n)\in K^n \to (\Delta_{n-1}(\underline{\hat{a_i}}))_{1\leq i\leq n}\in K^n$ and in the second we study a $K$-etale algebra defined by $n$ Artin-Schreier equations. In this context we express a group action in terms of an appropriate Elkies pairing.

\section{Generalities and motivations}

\subsection{Notations}

In this note all rings are commutative and unitary and $A$ (resp. $K$) denotes a ring (resp. a field) of characteristic $p>0$ containing the field  $\F_q$ where $q:=p^s$.   Finally 
$F: A\to A$ with $F(a)=a^q$, denotes the Frobenius endomorphism.

We denote by $K^{alg}$ a $K$ algebraic closure. 
We adopt the following notations when the context is not ambiguous.

Let  $n\geq 1$, $m\geq 0$ be integers.

$(\underline{a}):=(a_1,a_2,\cdots,a_n)$, with $a_i\in A$ and $n\geq 1$,

$(\underline{X}):=(X_1,X_2,\cdots,X_m)$ be indeterminates over $A$ and let $m\geq 0$ with the convention $\underline{X}=\emptyset $ if $m=0$. The integer $m$ is determined by the context.

%%%%%%%%%%%%%%%%%%%%%%%%%%%

$(\underline{\hat a_i},\underline{X}):=(a_1,\cdots,\hat a_i,\cdots,a_n,\underline{X})$, i.e. we omit $a_i$ and $\underline{X}$ may be empty. The integers $n$ and $m$ are determined by the context.
\begin{defi}\label{defi2.1}
Let $A$ be a commutative ring containing the finite field  $\F_q$. Let $m,n\geq 1$ be integers and $\underline{a}:=(a_1,\cdots,a_n)\in A^n$. We call
{\rm Moore matrix} of size $m,n$ associated to $\underline{a}$, the matrix 
of $M_{m,n}(A)$ denoted  by
${\cal{M}}_{m,n}(\underline{a})$, (${\cal{M}}_n(\underline{a})$ if $n=m$) where
$$
{\cal{M}}_{m,n}(\underline{a}):=\left(
\begin{array}{cccc}
a_1&a_2&\cdots&a_n   \\
a_1^q&a_2^q&\cdots&a_n^q   \\
\vdots&\vdots&\cdots&\vdots\\
a_1^{q^{m-1}}&a_2^{q^{m-1}}&\cdots&a_n^{q^{m-1}}  
\end{array}
\right)
$$
\msn
and {\rm Moore determinant} associated to $\underline{a}$, the determinant of ${\cal{M}}_n(\underline{a})$ denoted by
$
\Delta_n(\underline{a}).
$

\end{defi}

\subsection{Additive polynomials and Moore determinants } 

\begin{defi}\label{defi2.2}
 Let $K$ be a field containing $\F_q$. We call \  {\rm  $ \F_q $-linear 
 polynomial} a polynomial of the form $c_nX^{q^n}+c_{n-1}X^{q^{n-1}}+\cdots+
 c_iX^{q^i}+\cdots+c_0X\in K[X]$.

It is easy to see that a polynomial $P\in K[X]$ is a  $\F_q$-linear polynomial if and only if it satisfies the following two conditions.
\begin{enumerate}

 \item $P(X+Y)=P(X)+P(Y)$ in the polynomials ring $K[X,Y]$
 \item $P(\lambda X)=\lambda P(X)$  for all $\lambda\in \F_q$.
 
A polynomial is {\rm additive} if it satisfies condition 1. An additive polynomial is said to be  
{\rm reduced} if it is separable. If it is non-zero, this means that the coefficient of $X$ is non-zero.
\end{enumerate}

Let $P\in K[X]$ be a $\F_q$-linear polynomial, let $\Ker P:=
\{x\in K^{alg}\ |\ P(x)=0\}$ be the set  roots of $P$; it is a  $\F_q$-subspace of $K^{alg}$.

The application 
$x\in K \to P(x)\in K$ is a $\F_q$-linear endomorphism of $K$. Thus we can consider the $\F_q$-subspace of $K$ which is the {\rm kernel} of this endomorphism. It coincides with $\Ker P$ when $\Ker P\subset K.$

\end{defi}

\begin{prop}\label{prop2.1}(\cite{Mo}, \cite{Or} and  \cite{E} prop.1 p.80)
 Let $A$ be an integral commutative ring containing $\F_q$. The $n$ elements of $A$, $a_1,a_2,\cdots,a_n$ are $\F_q$-linearly independent if and only if $\Delta_n(\underline{a})\neq 0$ . In other words the $n$ elements $a_1,a_2,\cdots,a_n$ of $A$ are $\F_q$-linearly independent if and only if the $n$ vectors $\underline a,F(\underline a),\cdots,F^{n-1}(\underline a)$ of $A^n$, are $\F_q$-linearly independent. 
\end{prop}

This proposition is a consequence of Moore's identity 
(\cite{E} (3.4) p. 80, (3.6) p. 81) which says that if $a_1,a_2,\cdots,a_n$ are elements of $A$, then

\begin{eqnarray}
 \Delta_n(\underline{a})=\prod_{1\leq i\leq n}\prod_{\epsilon_{i-1}\in \F_q}\cdots\prod_{\epsilon_{1}\in \F_q}(a_i+\epsilon_{i-1}a_{i-1}+\cdots+\epsilon_{1}a_1)
\label{F1}
 \end{eqnarray}

\begin{prop}\label{prop2.2}
 Let $K$ be a field containing $\F_q$, $W$ a $\F_q$-vector subspace of $K$ with $\dim_{\F_q}W=n$, then there exists a unique unit polynomial of degree $q^n$, denoted $P_W$ with 
 $W=\{x\in K\ |\ P_W(x)=0\}$. Moreover $P_W$ is a $\F_q$-linear polynomial
 which is reduced and if $\underline{w}:=(w_1,\cdots,w_n)\in K^n$ is a $\F_q$-basis of $W$ then 
\begin{eqnarray}
P_W(X)=\prod_{w\in W}(X-w)=\frac{\Delta_{n+1}(\underline{w},X)}{\Delta_{n}(\underline{w})}=X^{q^n}+\cdots+(-1)^n\Delta_{n}(\underline{w})^{q-1}X
\label{F2}
\end{eqnarray}
\end{prop}

The following proposition is the version adapted to hyperplanes in Proposition \ref{prop2.2}.

\begin{prop}\label{prop2} Let $W$ be a $\F_q$-subspace of $K$ with $\dim_{\F_q}W=n$, $\underline{w}:=(w_1,\cdots,w_n)\in K^n$ a $\F_q$-basis of $W$ and $\underline{w^\star}:(w_1^\star,\cdots,w_n^\star)$ its dual basis.
Let $\varphi$ be a non-zero $\F_q$-linear form on $W$ and $\Ker\varphi$ the hyperplane of $W$ kernel of $\varphi$. Let  $\underline{\alpha}:=(\alpha_1,\cdots,\alpha_n)\in \F_q^n-\{(0,\cdots,0)\}$ such that $\varphi=\sum_{1\leq i\leq n}\alpha_iw_i^\star$ and 
$$\Delta_{\varphi}(\underline{w},X)=\left|
\begin{array}{ccccc}
\alpha_1&\alpha_2&\cdots&\alpha_n&0\\
w_1&w_2&\cdots&w_n&X   \\
w_1^q&w_2^q&\cdots&w_n^q &X^q  \\
\vdots&\vdots&\cdots&\vdots&\vdots\\
w_1^{q^{n-1}}&w_2^{q^{n-1}}&\cdots&w_n^{q^{n-1}}  &X^{q^{n-1}}

\end{array}
\right| 
\ ,\delta_{\varphi}(\underline{w}):=\left|
\begin{array}{cccc}
\alpha_1&\alpha_2&\cdots&\alpha_n\\
w_1&w_2&\cdots&w_n  \\
w_1^q&w_2^q&\cdots&w_n^q   \\
\vdots&\vdots&\cdots&\vdots\\
w_1^{q^{n-2}}&w_2^{q^{n-2}}&\cdots&w_n^{q^{n-2}}   

\end{array}
\right|$$

\noindent Then $\delta_{\varphi}(\underline{w})\neq 0$ and like in (\ref{F2}) we can write
\begin{eqnarray}
P_{\ker \varphi}=\prod_{w\in \ker \varphi} (X-w)=\frac{\Delta_{\varphi}(\underline{w},X)}{\delta_{\varphi}(\underline{w})}=X^{q^{n-1}}+\cdots+(-1)^{n+1}\delta_{\varphi}(\underline{w})^{q-1}X,\label{F3}
   \end{eqnarray}     
        it is a $\F_q$-linear polynomial of degree $q^{n-1}$ which is reduced. 
Moreover we have for $1\leq i\leq n,\ 
    \Delta_{w_i^\star}(\underline{w},X)
 =$
 
 \noindent $(-1)^{i-1}\Delta_n(\underline{\hat w_i},X),\ 
        \Delta_{\varphi}(\underline{w},X)=\sum_{1\leq i\leq n}\alpha_i\Delta_{w_i^\star}(\underline{w},X)=
\sum_{1\leq i\leq n}\alpha_i(-1)^{i-1}\Delta_n(\underline{\hat w_i},X)  \ {\rm and }\ \delta_{\varphi}(\underline{w})=$

\noindent $\sum_{1\leq i\leq n}(-1)^{i-1}\alpha_i\Delta_{n-1}(\underline{\hat w_i}).$
\end{prop} 

\begin{proof}
First we show that $\Delta_{\varphi}(\underline{w},X)$ is not the null polynomial. Let us assume the opposite. Since 
$\underline{w},F(\underline{w}),\cdots,F^{n-1}(\underline{w})$ are 
$\F_q$-linearly independent and since for $j\in \{0,\cdots,n-1\}$, the coefficient of $X^{q^j}$ is zero, it follows that $\underline{\alpha}\in \sum_{i\in \{0,\cdots,n-1\},\ i\neq j} \F_q F^i(\underline{w});$ thus its $(j+1)$-th coordinate in the $\F_q$-basis 
$\underline{w},F(\underline{w}),\cdots,F^{n-1}(\underline{w})$ is zero, which contradicts the non-nullity of  $\underline{\alpha}.$

The polynomial $\Delta_{\varphi}(\underline{w},X)$ is thus a $\F_q$-linear polynomial of degree $\leq q^{n-1}$. We check that it is zero on the hyperplane $\ker \varphi$; thus its degree is equal to $q^{n-1}$. Hence the proposition.
\end{proof}

\begin{coro}
 Let  $K$ be a field containing $\F_q$,   $\underline{w}:=(w_1,w_2\cdots,w_n)\in K^n$ and for  $1\leq i\leq n$, $\Delta_{n-1}(\underline{\hat w_i}):=\Delta_{n-1}(w_1,\cdots,w_{i-1},w_{i+1},\cdots,w_n)$, 
 then
 $\Delta_n(\underline{w})\neq 0$ if and only if $\Delta_n(\underline{\Delta_{n-1}(\underline{\hat w_i})})=$
 
 \noindent $ \Delta_n(\Delta_{n-1}(\underline{\hat w_1}),\Delta_{n-1}(\underline{\hat w_2}),\cdots,\Delta_{n-1}(\underline{\hat w_n}))\neq 0$.
\end{coro}

\begin{proof}
Let us assume that $ \Delta_n (\underline {w}) \neq 0 $, then by Proposition \ref {prop2.1} and  \ref {prop2}, $\delta_{\underline{\varphi}}(\underline{w})=\sum_{1\leq i\leq n}\alpha_i\Delta_{n-1}(\underline{\hat w_i})\neq 0$ for all
$\underline{\alpha}:=(\alpha_1,\alpha_2,\cdots,\alpha_n)\in \F_q^n-\{(0,\cdots,0)\}$
where $\varphi=\sum_{1\leq i\leq n}\alpha_iw_i^\star$. It follows from Proposition \ref{prop2.1} that $\Delta_n(\underline{\Delta_{n-1}(\underline{\hat w_i})})\neq 0$.

Let us assume that  $\Delta_n(\underline{w})= 0$,  then by Proposition \ref{prop2.1}, there is $(\epsilon_1,\epsilon_2,\cdots,\epsilon_n)\in \F_q^n-\{(0,\cdots,0)\}$ with $\sum_{1\leq i\leq n}\epsilon_i w_i=0$. Let $f$ be the $\F_q$-linear form over $\F_q^n$ such that $f((\alpha_1,\cdots,\alpha_n))=\sum_{1\leq i\leq n}\epsilon_i \alpha_i$ and $(\alpha_1,\cdots,\alpha_n)\in \ker f-\{0\}$, then $\left|
\begin{array}{cccc}
\alpha_1&\alpha_2&\cdots&\alpha_n\\
w_1&w_2&\cdots&w_n  \\
w_1^q&w_2^q&\cdots&w_n^q   \\
\vdots&\vdots&\cdots&\vdots\\
w_1^{q^{n-2}}&w_2^{q^{n-2}}&\cdots&w_n^{q^{n-2}}   

\end{array}
\right|=0$, so
$\sum_{1\leq i\leq n}(-1)^{i-1}\alpha_i\Delta_{n-1}(\underline{\hat w_i})=0$.

\noindent Thus $\Delta_n(\underline{\Delta_{n-1}(\underline{\hat w_i})})= 0.$ 
\end{proof}

\begin{defi}\label{Ndefi2.3}
Let $P(X):=c_nX^{q^n}+c_{n-1}X^{q^{n-1}}+\cdots+
 c_iX^{q^i}+\cdots+c_0X\in K[X]$ be a reduced  $\F_q$-linear polynomial
 (i.e.
$c_0\neq 0$) of degree $q^n$ (i.e. $c_n\neq 0$).
With Ore we consider 
the {\rm reversed polynomial} $\rho P$ of the polynomial $P$ where 
$$(\rho P)(X):=\sum_{0\leq m\leq n}c_{n-m}^{q^m}X^{q^m}.$$
It is a reduced  $\F_q$-linear polynomial of degree $q^n$.
\end{defi}

O. Ore shows the following  result (see \cite{E}  Theorem 5 p.88).

\begin{prop}\label{nprop2.3} Let $K$ be a field containing $\F_q$. Let $P =\sum_{0\leq i\leq n}c_iX^{q^i}\in K[X]$ be a reduced $\F_q$-linear polynomial of degree $q^n$, $\rho P$ its reversed polynomial (Definition \ref{Ndefi2.3}). We assume that the roots of $P$ are in $K$. Let 
$W:=Ker P\subset K$ (Definition \ref{defi2.2}) and $\underline{w}:=(w_1,w_2,\cdots,w_n)\in K^n$ be a $\F_q$-basis of $W$. Let $\hat W\subset K$, the $\F_q$-subspace of $K$ spanned
by the $n$ minors $\Delta_{n-1}(\underline{\hat w_i}),\ 1\leq i\leq n$, then if $U: =\Ker \rho P$, we have $U= c_n^{-1}(\frac{\hat W}{\Delta_n(\underline{w})})^q$; this is a $\F_q$-subspace of $K$ of 
dimension $n$.
 
\end{prop}

\begin{rem} 
It follows from Proposition \ref{nprop2.3} (see also \cite{G}, Corollary 1.7.14 p.18) that if the $n$ elements $w_1,\cdots,w_n$ in $K$ are $\F_q$-independent, so are the $n$ elements $\Delta_{n-1}(\underline{\hat w_i}),\ 1\leq i\leq n$.  Although not explicitly written, O. Ore (\cite{Or}) and N. Elkies (\cite{E}) show the following result

\end{rem}
\begin{prop}\label{Nprop2.4}
Let  $K$ be a field containing $\F_q$,   $\underline{w}:=(w_1,w_2,\cdots,w_n)\in K^n$ and for  $1\leq i\leq n$, $\Delta_{n-1}(\underline{\hat w_i}):=\Delta_{n-1}(w_1,\cdots,w_{i-1},w_{i+1},\cdots,w_n)$. Let us assume that
$\Delta_n(\underline w)\neq 0$.
Then
\begin{eqnarray}
(\Delta_n(\underline{\Delta_{n-1}(\underline{\hat w_i})}))^{q-1}=\Delta_n(\underline w)^{q^{n-1}-1}.
\label{F4}
\end{eqnarray}
Thus
\begin{eqnarray}
\frac{\Delta_n(\underline{\Delta_{n-1}(\underline{\hat w_i})})}{\Delta_n(\underline w)^{1+q+\cdots+q^{n-2}}}\in \F_q^\star.
\label{F5}
\end{eqnarray}
 \end{prop}

 \begin{proof}
 
 Let $W\subset K$ be the $\F_q$-vector space $\bigoplus_{1\leq i\leq n}\F_qw_i$ and $P_W$ the polynomial associated to $W$ by Proposition \ref{prop2.2}. Let $\hat W\subset K$ be the $\F_q$-vector space $\bigoplus_{1\leq i\leq n}\F_q\Delta_{n-1}(\underline{\hat w_i}).$ 

 With ({\ref{F2}}) we have  $P_W(X):=\frac{\Delta_{n+1}(\underline w,X)}{\Delta_n(\underline w)}$ and if we write 
 $\Delta[i](\underline w):= \det (\underline w,F(\underline w),\cdots,\hat F^i(\underline w),\cdots,F^{n}(\underline w))$ (i.e. the line 
$F^i(\underline w)$ is left out ),
then $P_W(X)
=\sum_{0\leq m\leq n}(-1)^{n-m}\frac{\Delta[m](\underline w)}{\Delta_n(\underline w)}X^{q^m}=:\sum_{0\leq m\leq n}c_mX^{q^m}$, thus $c_m=(-1)^{n-m}\frac{\Delta[m](\underline w)}{\Delta_n(\underline w)}$. 

Applying the above to the family $\underline{\Delta_{n-1}(\underline{\hat w_i})}$ and the polynomial 
$P_{\hat W}(X):=\frac{\Delta_{n+1}(\underline{\Delta_{n-1}(\underline{\hat w_i})} ,X)}{\Delta_n(\underline {\Delta_{n-1}(\underline{\hat w_i})})}=:\sum_{0\leq m\leq n}\hat{c}_mX^{q^m}$ we obtain the following identities  for $0\leq m\leq n$,
$\hat{c}_m=(-1)^{n-m}\frac{\Delta[m](\underline {\Delta_{n-1}(\underline{\hat w_i})})}{\Delta_n(\underline {\Delta_{n-1}(\underline{\hat w_i})})}$. In particular $\hat c_0=(-1)^n(\Delta_n(\underline {\Delta_{n-1}(\underline{\hat w_i})}))^{q-1}$.

Elkies (cf. \cite{E} 4.28) shows, following Ore that 
$$P_{\hat W}(X)=X^{q^n}+(-1)^n(\sum_{1\leq m\leq n-1}c_{n-m}^{q^{m-1}}\Delta_n(\underline w)^{q^{n-1}-q^m}X^{q^m}+\Delta_n(\underline w)^{q^{n-1}-1}X).$$
Thus \eqref{F4} is fulfilled.

In the case where $1\leq m\leq n-1$ we obtain the equality 

\msn
$\hat{c}_m=(-1)^{n-m}\frac{\Delta[m](\underline{\Delta_{n-1}(\underline{\hat w_i})})}{\Delta_n(\underline{\Delta_{n-1}(\underline{\hat w_i})})}=(-1)^nc_{n-m}^{q^{m-1}}\Delta_n(\underline w)^{q^{n-1}-q^m},$
which taking into account \eqref{F4} gives
\begin{eqnarray}
(\Delta[m](\underline{\Delta_{n-1}(\underline{\hat w_i})}))^{q-1}= 
\Delta_n(\underline w)^{q^{n}-q^{m+1}+q^{m-1}-1}(\Delta[n-m](\underline w))^{q^{m-1}(q-1)}.
\label{F6}
\end{eqnarray}
\end{proof}
\noindent
\begin{rem}\label{rem2.2}
 If we take into account Theorem {\ref{thm1}}, we can specify the equalities
 \eqref{F5} and \eqref{F6}. Thus we have
\begin{eqnarray}
\frac{\Delta_n(\underline{\Delta_{n-1}(\underline{\hat w_i})})}{\Delta_n(\underline w)^{1+q+\cdots+q^{n-2}}}=
(-1)^{{\left\lfloor\frac{n}{2}\right\rfloor}} 
\label{F7}
\end{eqnarray}
where $\left\lfloor\frac{n}{2}\right\rfloor$ is the lower integer part of $\frac{n}{2}$, and 

\begin{eqnarray}
\Delta[m](\underline{\Delta_{n-1}(\underline{\hat w_i})})=(-1)^{{\left\lfloor\frac{n}{2}\right\rfloor}} 
\Delta_n(\underline w)^{\frac{q^{n}-1}{q-1}-q^{m-1}-q^m}\Delta[n-m](\underline w)^{q^{m-1}}.
\label{F8}
\end{eqnarray}
\end{rem} 

The following proposition takes up results of Ore and Elkies  (\cite{E}, Proposition 3 ) which specify the link between 
composition of two  $\F_q$-linear polynomials and the geometry of sets of roots.

\begin{prop}\label{prop2.5} Let $K$ be a field which contains $\F_q$. Let $W$ be a $\F_q$-vector space of $K$ with $\dim_{\F_q}W=n$, $W_1$ be a $\F_q$-subvector space of $W$ and $P_W(X):=\prod_{x\in W}(X-x)$ (resp. $P_{W_1}(X):=\prod_{x\in W_1}(X-x)$). Then $P_{W_1}(W)$ is a 
finite dimensional $\F_q$-subspace of $K$  and
\begin{eqnarray}
  P_W(X)=P_{P_{W_1}(W)}(P_{W_1}(X))\label{F9}
 \end{eqnarray}
Conversely, if $Q$ is a monic  $\F_q$-linear polynomial such that $P_W(X)=Q(P_{W_1}(X))$, we have
$Q=P_{P_{W_1}(W)}.$
\end{prop}

\begin{proof} Let us assume that $\dim_{\F_q}W_1=m$ then $\deg P_{W_1}=q^m$. Since $x\in W\to P_{W_1}(x)\in K$ is a $\F_q$-linear map whose kernel is $W_1$, it follows that $P_{W_1}(W)$ is a $\F_q$-subspace of $K$ of dimension $n-m$; thus the polynomial 
$P_{P_{W_1}(W)}(P_{W_1}(X))$ is a monic $\F_q$-linear polynomial of degree $q^{n-m}q^m=q^n$. Since it is by construction zero on $W$, it follows that $P_W(X)$ divides  $P_{P_{W_1}(W)}(P_{W_1}(X))$ in $K[X]$, hence the equality.

For the reciprocal we remark that $(Q- P_{P_{W_1}(W)})(P_{W_1}(X))$ is the null polynomial in $K[X]$ and that $P_{W_1}(X)$ is transcendental over $K$.
 
\end{proof}

Finally, the following corollary specifies Proposition \ref{prop2}
\begin{coro}\label{Ncor2.2} Let $K$ be a field which contains $\F_q$. Let $W\subset K$ be a $\F_q$-vector space with $\dim_{\F_q}W=n$, $\underline{w}:=(w_1,w_2,\cdots,w_n)\in K^n$ a $\F_q$-basis of $W$ and $(w_1^\star,w_2^\star,\cdots,w_n^\star)$ its dual basis.
Let $\varphi$ be a non-zero $\F_q$-linear form on $W$ and $\Ker\varphi$  be the hyperplane of $W$ kernel of $\varphi$. Let  $\underline{\alpha}:=(\alpha_1,\alpha_2,\cdots,\alpha_n)\in \F_q^n-\{(0,0,\cdots,0)\}$ and such that $\varphi=\sum_{1\leq i\leq n}\alpha_iw_i^\star$,  $\Delta_{\varphi}(\underline{w},X)=
 \sum_{1\leq i\leq n}\alpha_i(-1)^{i-1}\Delta_n(\underline{\hat w_i},X) $
and  $\delta_{\varphi}(\underline{w})=\sum_{1\leq i\leq n}\alpha_i(-1)^{i-1}\Delta_{n-1}(\underline{\hat w_i}).$ Let $P_W\in K[X]$ (resp. $P_{\Ker\varphi}\in K[X]$) be the monic and reduced polynomial whose set of roots is  $W$ (resp. $\Ker\varphi$). 

Then \eqref{F9} is satisfied with
 $W_1:=\Ker\varphi$, 
and 
 $P_{\Ker\varphi}(X)=\frac{\Delta_\varphi(\underline{w},X) }{ \delta_\varphi(\underline{w})}$, 
$P_{P_{\Ker\varphi}(W)}(X)=X^q-(\frac{\Delta_n(\underline w)}{\delta_\varphi(\underline{w})})^{q-1}X, $

\noindent thus
$\frac{P_{\Ker\varphi}(X)}{P_W(X)}=\frac{\Delta_n(\underline w)}{\delta_\varphi(\underline{w})}\frac{\Delta_\varphi(\underline{w},X) }{\Delta_{n+1}(\underline{w},X) }=\frac{1}{P_{\Ker\varphi}(X)^{q-1}-(\frac{\Delta_n(\underline w)}{\delta_\varphi(\underline{w})})^{q-1}}$.
\end{coro}

\begin{proof} 
 Since $\Ker\varphi$ is a hyperplane of $W$, it follows that 
$P_{P_{\Ker\varphi}(W)}(X)=X^q-c_{\varphi} X\in K[X]$. Thus 
with \eqref{F9} we have $P_W(X)=P_{W_1}(X)^q-c_{\varphi}P_{W_1}(X)$.
Since $\coeff_XP_W=(-1)^n\Delta_{n}(\underline{w})^{q-1}$ (cf. \eqref{F2}) and (cf. \eqref{F3}) 
$\coeff_XP_{W_1}=(-1)^{n+1}\delta_{\varphi}(\underline{w})^{q-1}$, the corollary follows.
 
\end{proof}

\section{Vector spaces of differentials and Moore determinants}

\subsection{The $\F_q$-spaces $L_{\mu+1,n}^q$}

\begin{defi}\label{Lpagot}  Let $K$ be a field of characteristic $p>0$. Let $\mu\in \N$ with $\mu\geq 2$ prime to $p$ and $n\geq 1$. We call {\rm space $ L_{\mu + 1, n} $}   a $\F_p$-vector space of dimension $n$ of logarithmic differential forms in $\Omega^1_K(K(X))$, whose nonzero elements have $(\mu-1)\infty$ as zero divisor
  and their poles are in $K$ (such a form has $\mu+1$ poles and they are simple). One can show that if such a space exists then $p^{n-1}$ divides $\mu+1$.
\end{defi}

Such $\F_p$-vector spaces have been constructed for $n\geq 1$ in [Ma] in order in particular to lift in null characteristic certain $(\Z/p\Z)^n$-coverings of the projective line $ \pr_K$ into Galoisian coverings of group $(\Z/p\Z)^n$. See \cite{Ob} for a presentation of recent contributions on the subject.

The spaces $L_{\mu+1,n}$ have been defined and studied by Guillaume Pagot in his thesis (\cite{P1} p. 19), a part of which is published in \cite{P2}. See also \cite{MT} for complements.

We will consider a generalization to the case where $\F_p$ is replaced by the field $\F_q$ with $q=p^s$.

\begin{defi}\label{Ndef3.2} Let $K$ be a field which contains $\F_q$. 
 Let $\mu\in \N$ with $\mu\geq 2$ prime to $p$ and $n\geq 1$. Let $L_{\mu+1,n}^q$ be a $\F_q$-vector space of dimension $n$ of differential forms in $ \Omega^1_K(K(X))$ whose nonzero elements have $(\mu-1)\infty$ as zero divisor, such that their poles are simple and in $K$ (so there are $\mu+1$ poles) and such that their residues are in $\F_q$. 
\end{defi}

\begin{prop} \label{propMAT} Let $K$ be a field containing $\F_q$, $W$ a $\F_q$-subspace of $K$ with $\dim_{\F_q}W=n$, $\underline{w}:=(w_1,\cdots,w_n)\in K^n$ a $\F_q$-basis of $W$ and $\underline{w^\star}:(w_1^\star,\cdots,w_n^\star)$ its dual basis. For
$j\in \{1,\cdots,n\}$, we note  
$$
\omega_j:=\sum_{
(\epsilon_1,\epsilon_2,\cdots\epsilon_n)\in \F_q^n}
 \frac{\epsilon_jdX}{X-\sum_{i=1}^n\epsilon_iw_i}.
$$

Let  $\underline{\alpha}:=(\alpha_1,\cdots,\alpha_n)\in \F_q^n-\{(0,\cdots,0)\},\ \varphi=\sum_{1\leq i\leq  n}\alpha_iw_i^\star$ and  $$\omega_\varphi:=\sum_{1\leq j\leq  n}\alpha_j\omega_j=\sum_{
(\epsilon_1,\epsilon_2,\cdots\epsilon_n)\in \F_q^n}
 \frac{\sum_{1\leq j\leq n}\alpha_j\epsilon_j}{X-\sum_{i=1}^n\epsilon_iw_i}dX.$$
Then,
$\omega_\varphi= -\Delta_{n}(\underline{w})^{q-1}\frac{\Delta_\varphi(\underline{w},X)}{ \Delta_{n+1}(\underline{w},X)}dX$ and
$\Delta_\varphi(\underline{w},X)\  |\ \Delta_{n+1}(\underline{w},X)$  (cf. Proposition \ref{prop2}).

Thus $\Omega_W:=\sum_{1\leq j\leq n}\F_q\omega_j\subset \Omega^1_K(K(X))$ is a 
$L_{\mu +1,n}^q$ with $\mu+1:=q^{n-1}(q-1)$. 
\end{prop}

\begin{proof}
 Let $\omega\in \Omega_W$ and  $\underline{\alpha}:=(\alpha_1,\alpha_2,\cdots,\alpha_n)\in \F_q^n-\{(0,0,\cdots,0)\}$ with $\omega=\sum_{1\leq j\leq  n}\alpha_j\omega_j$. Then
 $\omega=\sum_{
(\epsilon_1,\epsilon_2,\cdots,\epsilon_n)\in \F_q^n}
 \frac{\sum_{1\leq j\leq n}\alpha_j\epsilon_j}{X-\sum_{i=1}^n\epsilon_iw_i}dX$.It follows that the poles of $\omega$ are the elements of $W$ deprived of the zeros of the $\F_q$-linear form $\varphi=\sum_{1\leq i\leq n}\alpha_iw_i^\star$ so of cardinal 
 $q^n-q^{n-1}=:\mu+1$ and they are simple. The residues by construction are in $\F_q$. In particular $\omega\neq 0$ and therefore 
 $\Omega_W$ is a $\F_q$-vector space of dimension $n$.
 
 It remains to see that the zero divisor of $\omega$ is $(\mu-1)\infty$. For that we consider the fraction 
 $F(X):=-\Delta_{n}(\underline{w})^{q-1}\frac{\Delta_\varphi(\underline{w},X)}{ \Delta_{n+1}(\underline{w},X)}$.  The poles of $F$ are the elements of $W-\ker \varphi$ and they are simple (Corollary  \ref{Ncor2.2}).
 
 Let $w\in W$ with $\varphi(w)\neq 0$,
 then $\res_wF(X)=-\Delta_{n}(\underline{w})^{q-1}\frac{\Delta_\varphi(\underline{w},w)}{ \Delta_{n+1}(\underline{w},X)'(w)}$. We have $w=\sum_{i=1}^n\epsilon_i(w)w_i$ with
 $\epsilon_i(w)\in \F_q$; thus (cf. Proposition \ref{prop2}) 
 
 $\Delta_\varphi(\underline{w},w)=\sum_{1\leq i\leq n}\alpha_i(-1)^{i-1}\Delta_n(\underline{\hat w_i},w)=\sum_{1\leq i\leq n}\alpha_i(-1)^{i-1}\epsilon_i(w)\Delta_n(\underline{\hat w_i},w_i)=$
 
 \noindent$\sum_{1\leq i\leq n}\alpha_i(-1)^{i-1}\epsilon_i(w)(-1)^{n-i}\Delta_n(\underline{w})=(-1)^{n-1}(\sum_{1\leq i\leq n}\alpha_i\epsilon_i(w))\Delta_n(\underline{w}).$
 
 \ms Finally (cf. Proposition \ref{prop2.2})) $\Delta_{n+1}(\underline{w},X)'(w)=(-1)^n\Delta_{n}(\underline{w})^{q}.$
 
 Thus $\res_wF(X)=-\Delta_{n}(\underline{w})^{q-1}\frac{ (-1)^{n-1}(\sum_{1\leq i\leq n}\alpha_i\epsilon_i(w))\Delta_n(\underline{w}) }{(-1)^n\Delta_{n}(\underline{w})^{q} }=\sum_{1\leq i\leq n}\alpha_i\epsilon_i(w)$, 
 and so $\omega=\omega_\varphi$. It follows that the zeros of $\omega$ are concentrated at infinity (Corollary  \ref{Ncor2.2}).

\end{proof}

\begin{rem}
In (\cite{P1} Remark 4 p.29) Pagot remarks that if $K$ is algebraically closed then the pullback by a morphism 
$\Phi :\pr_K\to \pr_K$ with  $\Phi(X)=\alpha X+X^pP(X^p)$ where $\alpha\in K^\star$ and $P\in K[X]$,  of a space 
$L_{\mu+1,n}$ is a space $L_{(\mu+1)\deg \Phi,n}$. Similarly an exercise shows that the pullback of a $L^q_{\mu+1,n}$ is a space $L^q_{(\mu+1)\deg \Phi,n}$. We can thus construct new spaces $L^q_{\mu+1,n}$ for example from Proposition \ref{propMAT}.

\end{rem}

\subsection{Spaces $L_{\mu+1,n}^q$ and Elkies pairing}

In this section $K$ denotes a field that contains $\F_q$, $W$ is a $\F_q$-vector space of $K$ with $\dim_{\F_q}W=n$, $\underline{w}:=(w_1,. \cdots,w_n)\in K^n$, a $\F_q$-basis of $W$ and $\underline{w^\star}:=(w_1^\star,\cdots,w_n^\star)$ its dual basis. Finally let
$\hat W:=\bigoplus_{1\leq i\leq n}\F_q \Delta_{n-1}(\underline{\hat w_i})$ and  $U:=(\frac{\hat W}{\Delta_n(\underline{w})})^q$. Recall that $U=Ker \rho P_W$ where 
$\rho P_W$ is the reversed polynomial of the polynomial $P_W$ (cf. Proposition \ref{prop2.2}, Definition \ref{Ndefi2.3}, Proposition \ref{nprop2.3}).

We will recall the construction of Elkies pairing  attached to the monic $\F_q$-linear polynomial $P_W$. It puts in duality the two $\F_q$-subvector spaces of $K$ which are $W$ and $U$ and we will give a differential interpretation of it using the spaces $L_{\mu+1,n}^q$ defined in Proposition {\ref{propMAT}}.

\msn
{\bf A. Elkies pairing } 

Elkies (\cite{E}, § 4.35 ) 
defines a  $\F_q$-perfect pairing $E: W\times U\to \F_q$ as follows. He first observes that if $(w,u)\in W\times U$ then $0=w((\rho P_W)(u))^{q^{-n}}-uP_W(w)=E(w,u)-E(w,u)^q$ where
$E(w,u):=\sum_{1\leq m\leq n}\sum_{0\leq j\leq m-1}((c_mu)^{q^{-m}}w)^{q^j}$, and $P_W(X)=\sum_{0\leq m\leq n}c_mX^{q^m}$ and $c_n=1$. It follows that $E(w,u)\in \F_q$. 

\msn
{\bf B. } The pairing $f: W\times U\to \F_q$. 

We will see that Proposition {\ref{propMAT}} allows to define a $\F_q$-pairing
$f:W\times U\to \F_q$ given by the residue of differential forms. 

Precisely,
if $w\in W$ and $u\in U$ are different from $0$, we can write (cf. Proposition {\ref{prop2}})
\begin{eqnarray}
u=(\frac{\delta_{\varphi}(\underline{w})}{\Delta_n(\underline{w})})^q \ {\rm with}\ \delta_{\varphi}(\underline{w})=\sum_{1\leq i\leq n}\alpha_i (-1)^{i-1}\Delta_{n-1}(\underline{\hat w_i})\label{F10}
\end{eqnarray}
where  $\underline{\alpha}:=(\alpha_1,\cdots,\alpha_n)\in \F_q^n-\{(0,\cdots,0)\}$ and $\varphi=\sum_{1\leq i\leq n}\alpha_iw_i^\star$.

Let  $\Omega_{\hat W}:=\sum_{1\leq j\leq n}\F_q\omega_j\subset \Omega^1_K(K(X))$, be the $\F_q$-space 
$L_{\mu+1,n}^q$  with  $\mu+1:=q^{n-1}(q-1)$ as defined in
Proposition  {\ref{propMAT}} and let
 $$\omega_{\varphi}:=\sum_j\alpha_j\omega_j=\sum_{
(\epsilon_1,\epsilon_2,\cdots\epsilon_n)\in \F_q^n}
 \frac{\sum_{1\leq j\leq n}\alpha_j\epsilon_j}{X-\sum_{i=1}^n\epsilon_iw_i}dX\in \Omega_{\hat W}-\{0\}.$$
 We can write  $w=\sum_{i=1}^n\epsilon_i(\underline{w},w)w_i$ in the basis  $\underline w$ of $W$ where $(\epsilon_1(\underline{w},w),\cdots,\epsilon_n(\underline{w},w))\in \F_q^n-\{0\}$. Then we define
\begin{eqnarray}
  f(w,u):=(-1)^{n-1}\res_w\omega_{\varphi}=(-1)^{n-1}\sum_{1\leq j\leq n}\alpha_j\epsilon_j(\underline{w},w)\in \F_q.\label{F11}
\end{eqnarray}

\begin{lem}
 The pairing $f$ is perfect.
\end{lem}

\begin{proof}
 Let $u\in U$, let us assume that $f(w,u)=0$ for all $w\in W$.  It follows from {\eqref{F10}} that $ u=(\frac{\delta_{\varphi}(\underline{w})}{\Delta_n(\underline{w})})^q \ {\rm with}\ \delta_{\varphi}(\underline{w})=\sum_{1\leq i\leq n}\alpha_i (-1)^{i-1}\Delta_{n-1}(\underline{\hat w_i})$
where  $\underline{\alpha}:=(\alpha_1,\cdots,\alpha_n)\in \F_q^n$ and $\varphi=\sum_{1\leq i\leq n}\alpha_iw_i^\star$.

 Thus for any $ w\in W$, the residue at $w$ of the differential form $\omega_{\alpha}=\sum \alpha_i\omega_i$ is zero and since the poles of $\omega_{\alpha}$ are simple, this form is the null form.
 Thus $\alpha_i=0$ for $1\leq i\leq n$, it follows that $\delta_{\varphi}(\underline{w})$ and thus $u$ are zero.

Now let $w\in W$, let us assume that $f(w,u)=0$ for all $u$ in $\hat W$. It follows from Proposition {\ref{propMAT}} that $0$ is the only element of $W$ which is not a pole of one of the $\omega_i$ forms, so $w=0$.
 \end{proof}
 
\noindent {\bf C. }Comparison of the two pairings $E$ and $f$

\begin{prop}\label{propAcc} The two pairings $E$ and $f$ are equal.
 
\end{prop}

\begin{proof}
Let $w_1,w_2,\cdots,w_n$ be a basis of $W$,  
 $w\in W-\{0\}$ and
$u\in U-\{0\}$. It follows from {\eqref{F10}} that

\noindent  $u=(\frac{\delta_{\varphi}(\underline{w})}{\Delta_n(\underline{w})})^q$ where $\underline{\alpha}:=(\alpha_1,\cdots,\alpha_n)\in \F_q^n-\{(0,\cdots,0)\}$ and
 $\varphi=\sum_{1\leq i\leq n}\alpha_iw_i^\star$.

 As from Proposition \ref{propMAT} we have $\omega_{\varphi}= -\Delta_{n}(\underline{w})^{q-1}\frac{\Delta_\varphi(\underline{w},X)}{ \Delta_{n+1}(\underline{w},X)}dX $,
it follows from \eqref{F3} and Corollary  \ref{Ncor2.2} that
$f(w,u)=(-1)^{n-1}\res_w\omega_\varphi=(-1)^n\Delta_{n}(\underline{w})^{q-1}\frac{\delta_\varphi(\underline{w}) P_{\ker \varphi}(w) } { (-1)^n\Delta_{n}(\underline{w})\Delta_{n}(\underline{w})^{q-1} }=u^{1/q}P_{\ker \varphi}(w)$ where $P_{\ker \varphi}(X)$ is a monic polynomial of  degree $q^{n-1}$ (cf. \eqref{F3}).
On the other hand 
$E(w,u)=u^{1/q}P_u(w)$ where $P_u(X)$ is a monic $\F_q$-linear polynomial of degree $q^{n-1}$ (\cite{E} lemma p. 92, proof).

It then remains to compare the two polynomials,  $P_{\ker \varphi}(X)$ and $P_u(X)$. As $P_{\ker \varphi}(X)$ and $P_u(X)$ divide $P_W(X)$ in $K[X]$ and as $\ker P_{\ker \varphi}$ (resp. $\ker P_u$) is an hyperplane in $W$ 
we have (Proposition {\ref{prop2.5}}) $P_W=Q_{\varphi}\circ P_{\ker \varphi}=Q_u\circ P_u$ in $K[X]$ where 
$Q_{\varphi}(X)=X^q-q_{\varphi}X$ and $Q_u=X^q-q_uX$ with $q_{\varphi},q_u$ non zero elements in $K$. In particular 
$\coeff_X(P_W)=-q_{\varphi}\coeff_X(P_{\ker \varphi})=-q_u\coeff_X(P_u)$. We have $\coeff_X(P_u)=-(-1)^n\Delta_n(\underline{w})^{q-1}u^{\frac{q-1}{q}}$ (\cite{E} lemma p. 92, proof). From  Corollary \ref{Ncor2.2} we know that  $q_{\varphi}=(\frac{\Delta_n(\underline w)}{\delta_\varphi(\underline{w})})^{q-1} $ and $\coeff_X(P_W)=  (-1)^n\Delta_{n}(\underline{w})^{q-1}$ (cf.
\eqref{F2}) and so $\coeff_X(P_{\ker \varphi})=\coeff_X(P_u)$; hence $q_{\varphi}=q_u$.

The equality $Q_\varphi=Q_u$ then follows from the uniqueness of the decomposition in Proposition \ref{prop2.5}.
\end{proof}

\section{A property of $L_{\mu+1,n}^q$ spaces}

\subsection{The property}

\begin{prop}\label{prop2.8} Let $\Omega$ be a $L_{\mu+1,n}^q$ space (définition 
\ref{Ndef3.2}) and $\underline\omega:=\omega_1,\omega_2,\cdots,\omega_n$ a $\F_q$-basis.
Let ${\cal P}(\Omega)\subset K$ be the set of poles of the differentials in  $\Omega$ and $P(X):=\prod_{x\in {\cal P}(\Omega)}(X-x)$. Let $P_i\in K[X]$ with
$\omega_i=\frac{P_i}{P}dX$ for $1\leq i\leq n$. Then there is $\gamma\in  K^\star$ with 
\begin{eqnarray}
  \Delta_n(P_1,\cdots,P_n)=\gamma P^{1+q+\cdots+q^{n-2}},\label{F12}
\end{eqnarray}
\end{prop}

\begin{proof}
Thanks to {\eqref{F1}} we can write
$$\Delta_n(P_1,\cdots,P_n)=\prod_{1\leq i\leq n}\prod_{\epsilon_{i-1}\in \F_q}\cdots\prod_{\epsilon_{1}\in \F_q}(P_i+\epsilon_{i-1}P_{i-1}+\cdots+\epsilon_{1}P_1).$$
By hypothesis each factor $P_i+\epsilon_{i-1}P_{i-1}+\cdots+\epsilon_{1}P_1$ divides $P$ and factorizes into a product of distinct irreducible polynomials of degree $1$. Thus for $x\in 
{\cal P}(\Omega)$, we must show that $x$ is a root of $1+q+\cdots+q^{n-2}$ polynomials $P_i+\epsilon_{i-1}P_{i-1}+\cdots+\epsilon_{1}P_1$ with $(\epsilon_1,\epsilon_2,\cdots,\epsilon_{i-1},1,0,0\cdots,0)\in \F_q^n$.

Let $x\in {\cal P}(\Omega)$, since $x$ is a pole of at least one of the $\omega_i$ forms,  the uplet 
$(\res_x\omega_i)_{1\leq i\leq n}\in \F_q^n-\{(0,\cdots,0)\}$ and let $\varphi_x: \Omega\to \F_q$ be the linear form with $\varphi_x(\omega)=\sum_{1\leq i\leq n}\alpha_i\res_x\omega_i$ where $\omega=\sum_{1\leq i\leq  n}\alpha_i \omega_j$,  then $\varphi_x$ is a $\F_q$-linear non zero form. If $ (\alpha_1:\alpha_2:\cdots:\alpha_n)\in \mathbb{P}^n(\F_q)$, $\varphi_x(\omega)=\sum_{1\leq i\leq n}\alpha_i\res_x\omega_i=0$ if and only if $\sum_{1\leq i\leq n}\alpha_iP_i(x)=0$, then as
  $\{(\epsilon_1,\epsilon_2,\cdots,\epsilon_{i-1},1,0,0\cdots,0),1\leq i\leq n\}\in \F_q^n$  is a system of representatives of the elements of $\mathbb{P}^n(\F_q)$, the multiplicity of the zero $x$ in $\Delta_n(P_1,\cdots,P_n)$ is equal to the number of points of the hyperplane of $\mathbb{P}^n(\F_q)$ induced by $Ker\varphi_x$, so it is equal to $1+q+\cdots+q^{n-2}$.

Finally by Proposition \ref{prop2.1}, $\gamma$ is not zero since the $n$ fractions $\frac{P_i}{P}$ are $\F_q$-linearly independent.

\end{proof}

\begin{rem}
The Proposition \ref{prop2.8} is a remark in \cite{P2} on page 68 in the framework of $L_{\mu+1,n}$-spaces of logarithmic differentials (i.e. $q=p$).
 
\end{rem}

\subsection{An equality between Moore's determinants I}

\begin{coro}\label{coro2.3}
 
 Let $(\underline{Y}):=(Y_1,Y_2,\cdots,Y_n)$ and
$(\underline{X}):=(X_1,X_2,\cdots,X_m)$, $n+m$ indeterminates over $\F_q$ where $n\geq 2,\ m\geq 0$ and the convention that $\underline{X}=\emptyset $  and $\Delta_m(\underline{X})=1$ if $m=0$. 
For $1\leq i\leq n$, we write $(\underline{\hat Y_i}):=(Y_1,\cdots,Y_{i-1},Y_{i+1}\cdots,Y_n).$. Then we have the following equality in  $\F_q(\underline Y,\underline X)$ 

$$\frac{\Delta_n(\Delta_{n-1+m}(\underline{\hat Y_1},\underline{X}),\cdots,(-1)^{i+1}\Delta_{n-1+m}(\underline{\hat Y_i},\underline{X}),\cdots,(-1)^{n+1}\Delta_{n-1+m}(\underline{\hat Y_n},\underline{X}))}{
\Delta_m(\underline{X})^{q^{n-1}}\Delta_{n+m}(\underline{Y},\underline{X})^{1+q+\cdots+q^{n-2}}}=$$
\begin{eqnarray}
\frac{\Delta_n(\Delta_{n-1}(\underline{\hat Y_1}),\cdots,(-1)^{i+1}\Delta_{n-1}(\underline{\hat Y_i}),\cdots,(-1)^{n+1}\Delta_{n-1}(\underline{\hat Y_n})}{
\Delta_{n}(\underline{Y})^{1+q+\cdots+q^{n-2}}}=:\gamma \label{13}
\end{eqnarray}

\end{coro}

\begin{proof}
 
Let $A:=\F_q[\underline Y,\underline X]$ and $K$ be its fraction field. For
$j\in \{1,\cdots,n+m\}$, we denote  
$$
\omega_j:=\sum_{
(\epsilon_1,\epsilon_2,\cdots\epsilon_{n+m})\in \F_q^{n+m}}
 \frac{\epsilon_jdZ}{Z-\sum_{i=1}^{n}\epsilon_iY_i-\sum_{i=1}^{m}\epsilon_{n+i}X_{i}},
$$
\noindent then $\Omega_{\underline Y,\underline X}:=\sum_{1\leq j\leq n+m}\F_q\omega_j\subset \Omega^1_K(K(Z))$ is a
$L_{\mu+1,n+m}^q$ space where $\mu+1:=q^{n-1+m}(q-1)$ and
$\Omega_{\underline Y}:=\sum_{1\leq j\leq n}\F_q\omega_j\subset \Omega^1_K(K(Z))$ is a $n$-dimensional $\F_q$-subspace $\Omega$ of $\Omega_{\underline Y,\underline X}$, hence it is a
$L_{\mu+1,n}^q$ space. 

We apply Proposition \ref{propMAT} to this last space $\Omega$. 

\noindent For $1\leq i\leq n$, we have
$\omega_i= -\Delta_{n+m}(\underline Y,\underline X)^{q-1}\frac{\Delta_{n+m}(\underline{\hat Y_i},\underline X,Z)}{ \Delta_{n+m+1}(\underline Y,\underline X,Z)}dZ$.  It follows that ${\cal P}(\Omega)=\{\sum_{i=1}^{n}\epsilon_iY_i+\sum_{i=1}^{m}\epsilon_{n+i}X_{i}\}$ with $(\epsilon_1,\epsilon_2,\cdots\epsilon_{n+m})\in \F_q^{n+m}$ and $(\epsilon_1,\epsilon_2,\cdots,\epsilon_n)\neq (0,\cdots,0)$ is the set of poles ${\cal P}(\Omega)\subset K$ of the elements of $\Omega$  .  

Thus $P(Z):=\prod_{z\in {\cal P}(\Omega)}(Z-z)=\frac{\Delta_{m}(\underline{X})}{\Delta_{n+m}(\underline{Y},\underline{X})}\frac{\Delta_{n+m+1}(\underline{Y},\underline{X},Z)}{\Delta_{m+1}(\underline{X},Z)}$ (cf. \eqref{F2}) and $\omega_i=\frac{P_i}{P} dZ$ where 
$\frac{P_i}{P}=-\Delta_{n+m}(\underline Y,\underline X)^{q-1}\frac{\Delta_{n+m}(\underline{\hat Y_i},\underline X,Z)}{ \Delta_{n+m+1}(\underline Y,\underline X,Z)}$. 
The equality \eqref{F12} in Proposition \ref{prop2.8}
then gives the following equality

\msn $\Delta_n(\Delta_{n+m}(\underline{\hat Y_1},\underline{X},Z),\cdots,(-1)^{i+1}\Delta_{n+m}(\underline{\hat Y_i},\underline{X},Z),\cdots,$
\begin{eqnarray}(-1)^{n+1}\Delta_{n+m}(\underline{\hat Y_n},\underline{X},Z))=
\gamma \Delta_{m+1}(\underline X,Z)^{q^{n-1}}\Delta_{n+m+1}(\underline{Y},\underline{X},Z)^{1+q+\cdots+q^{n-2}}\label{N13}
\end{eqnarray}
\noindent where $\gamma\in \F_q(\underline Y,\underline X)$.

\ms Finally the comparison in the equality \eqref{N13} of the coefficients of higher degree in $Z$ gives the equality 
\msn $\Delta_n(\Delta_{n-1+m}(\underline{\hat Y_1},\underline{X}),\cdots,(-1)^{i+1}\Delta_{n-1+m}(\underline{\hat Y_i},\underline{X}),\cdots,
(-1)^{n+1}\Delta_{n-1+m}(\underline{\hat Y_n},\underline{X}))=$

\hfill $\gamma \Delta_{m}(\underline X)^{q^{n-1}}\Delta_{n+m}(\underline{Y},\underline{X})^{1+q+\cdots+q^{n-2}}.
$

\ms By making $X_m$ play the role played by $Z$ in \eqref{N13} we deduce that

\msn $\Delta_n(\Delta_{n+m-2}(\underline{\hat Y_1},X_1,\cdots,X_{m-1}),\cdots,(-1)^{i+1}\Delta_{n-1+m-1}(\underline{\hat Y_i},X_1,\cdots,X_{m-1}),\cdots,$

\hfill $(-1)^{n+1}\Delta_{n-1+m-1}(\underline{\hat Y_n},X_1,\cdots,X_{m-1})=$

\hfill $\gamma \Delta_{m-1}(X_1,\cdots,X_{m-1})^{q^{n-1}}\Delta_{n-1+m-1}(\underline{Y},X_1,\cdots,X_{m-1})^{1+q+\cdots+q^{n-2}}$.

\ms By iterating the process we exhaust $\underline X$ hence 

\centerline{$\Delta_n(\Delta_{n-1}(\underline{\hat Y_1}),\cdots,(-1)^{i+1}\Delta_{n-1}(\underline{\hat Y_i}),\cdots,(-1)^{n-1}\Delta_{n-1}(\underline{\hat Y_n}))=\gamma \Delta_{n}(\underline{Y})^{1+q+\cdots+q^{n-2}}$ }

\noindent as announced.

\end{proof}

\subsubsection{An equality between Moore's determinants II}
We show in two different ways that the constant $\gamma$ in \eqref{13} is equal to $1$. 

Thus we can state the theorem. 
%\newpage
\begin{thm} \label{thm1} 
Let $(\underline{Y}):=(Y_1,Y_2,\cdots,Y_n)$ and
$(\underline{X}):=(X_1,X_2,\cdots,X_m)$ be  $n+m$ indeterminates over $\F_q$ where $n\geq 2,\ m\geq 0$ and the convention that $\underline{X}=\emptyset $  and $\Delta_m(\underline{X})=1$ for $m=0$. 
We write $(\underline{\hat Y_i}):=(Y_1,\cdots,Y_{i-1},Y_{i+1}\cdots,Y_n)$ for $1\leq i\leq n$. 

Then we have the following polynomial equalities in $\F_q[X,Y]$,

\noindent $ \tilde\Delta_{n,m}(\underline{Y},\underline{X}):=\Delta_n(\Delta_{n-1+m}(\underline{\hat Y_1},\underline{X}),\cdots,(-1)^{i+1}\Delta_{n-1+m}(\underline{\hat Y_i},\underline{X}),\cdots,(-1)^{n+1}\Delta_{n-1+m}(\underline{\hat Y_n},\underline{X}))=\hfill $
\begin{eqnarray}\qquad \qquad \qquad \qquad \qquad \qquad \qquad \qquad \qquad  \Delta_m(\underline{X})^{q^{n-1}}\Delta_{n+m}(\underline{Y},\underline{X})^{1+q+\cdots+q^{n-2}}\label{F13}
\end{eqnarray}

\noindent which is also (compare to \eqref{F12})

\noindent $ \Delta_n(\frac{\Delta_{n-1+m}(\underline{\hat Y_1},\underline{X})}{\Delta_m(\underline{X})},\cdots,\frac{\Delta_{n-1+m}((\underline{\hat Y_i},\underline{X})}{\Delta_m(\underline{X})},\cdots,\frac{\Delta_{n-1+m}((\underline{\hat Y_n},\underline{X})}{\Delta_m(\underline{X})})=$
\begin{eqnarray} \qquad \qquad \qquad \qquad \qquad \qquad \qquad \qquad \qquad 
(-1)^{\left\lfloor\frac{n}{2}\right\rfloor}(\frac{\Delta_{n+m}(\underline{Y},\underline{X})}{\Delta_m(\underline{X})})^{1+q+\cdots+q^{n-2}},\label{F14}
\end{eqnarray}
where $\left\lfloor\frac{n}{2}\right\rfloor$ is the lower integer part of $\frac{n}{2}$. We remark that  $\Delta_m(\underline{X})$ divides $\Delta_{n+m}(\underline{Y},\underline{X})$ thanks to (\ref{F1}).

\end{thm}

We deduce by specialization of formula {\eqref{F13}} the corollary
\begin{coro}
Let $A$ be a commutative ring containing $\F_q$.  Let $(\underline{a}):=(a_1,a_2,\cdots,a_n)\in A^n$ and
$(\underline{b}):=(b_1,b_2,\cdots,b_m)\in A^m$ where $n\geq 2,\ m\geq 0$ and the convention that $\underline{b}=\emptyset $  and $\Delta_m(\underline{b})=1$ for $m=0$. Then

\noindent $\Delta_n(\Delta_{n-1+m}(\underline{\hat a_1},\underline{b}),\cdots,(-1)^{i+1}\Delta_{n-1+m}(\underline{\hat a_i},\underline{b}),\cdots,(-1)^{n+1}\Delta_{n-1+m}(\underline{\hat a_n},\underline{b}))= $

$\qquad  \qquad  \qquad  \qquad  \qquad  \qquad \qquad  \qquad  \qquad  \qquad  \qquad  \qquad  \qquad  \qquad   \Delta_m(\underline{b})^{q^{n-1}}\Delta_{n+m}(\underline{a},\underline{b})^{1+q+\cdots+q^{n-2}}
$
\end{coro}

\subsection{First proof of Theorem \ref{thm1}. The case $m=1$ by induction on $n$. }

\noindent
{\bf A.} We check \eqref{F13} for $(n,m)=(2,1)$.  i.e.
 $\Delta_2(\frac{\Delta_2(Y_2,X)}{X},-\frac{\Delta_2(Y_1,X)}{X})=\frac{\Delta_3(Y_1,Y_2,X)}{X}.$ 
 
 This is an equality between polynomials in the variable $X$ of degree $q^2-1$. The terms of higher degree are equal since $\Delta_2(\Delta_1(Y_2X^{q-1}),-\Delta_1(Y_1X^{q-1}))=\Delta_2(Y_1,Y_2)X^{q^2-1}$. For $Y_{\underline{\alpha}}:=\alpha_1Y_1+\alpha_2Y_2$ with $\underline{\alpha}\in \F_q^2$ we have
 $\Delta_2(\Delta_2(Y_2,Y_{\underline{\alpha}}),\Delta_2(Y_1,Y_{\underline{\alpha}}))=0$, thus the two polynomials have the same zeros (see (\ref{F2})). Hence the equality.
 
\msn
 {\bf B.} We assume that $m=1$ and we  proceed by induction on $n$. Let us assume that 
 \eqref{F13}  is satisfied for $m=1$ and up to rank $n$. We show it for $m=1$ and $n+1$.
 
 \ms
 {\bf B1.} We show the equality of the coefficients of highest degree in \eqref{F13} for $m=1$ and $n+1$; it is also \eqref{F13} for $m=0$ and $n+1$.
 
 Let $(\underline{Y}):=(Y_1,Y_2,\cdots,Y_{n+1})$ be  $n+1$ indeterminates over $\F_q$.
 Let $j$ such that $1\leq j\leq n+1$,  we apply \eqref{F13} to the $n$ 
 indeterminates   $(Y_1,\cdots,Y_{j-1},\hat Y_j,Y_{j+1},\cdots,Y_{n+1}):=\underline{\hat Y_j}$ over $\F_q$ and we specialize $X$ in $Y_j$. Thus
 
\msn 
 $
\Delta_n(\Delta_n(\hat Y_1,...,\hat Y_j,..,Y_{n+1},Y_j),\cdots,\Delta_n(Y_1,...,\hat Y_i,..,\hat Y_j,..,Y_{n+1},Y_j),\cdots,
\Delta_n(Y_1,..,\hat Y_j,...,Y_n,\hat Y_{n+1},Y_j)) =$

\hfill $(-1)^{\left\lfloor\frac{n}{2}\right\rfloor}\tilde\Delta_{n,1}(\underline{\hat Y_j},Y_j)=\ \ (cf. \eqref{F13})$

\noindent$(-1)^{\left\lfloor\frac{n}{2}\right\rfloor}Y_j^{q^{n-1}}\Delta_{n+1}(Y_1,..,\hat Y_j,..,Y_{n+1},Y_j)^{1+q+\cdots+q^{n-2}}
 =(-1)^{\left\lfloor\frac{n}{2}\right\rfloor}Y_j^{q^{n-1}}((-1)^{n+1-j}\Delta_{n+1}(\underline{Y}))^{1+q+\cdots+q^{n-2}}$
\begin{eqnarray}
\qquad  \qquad \qquad \qquad \qquad \qquad \qquad \qquad =(-1)^{\left\lfloor\frac{n}{2}\right\rfloor}(-1)^{(n-1)(n+1-j)}Y_j^{q^{n-1}}\Delta_{n+1}(\underline{Y})^{1+q+\cdots+q^{n-2}}.\label{F15}
\end{eqnarray}
\ms
In what follows we use the following three identities:

\msn For $1\leq j<i\leq n+1$, we have $\Delta_n(Y_1,\cdots,\hat Y_j,\cdots,Y_i,\cdots,Y_{n+1})=(-1)^{n+1-i}
\Delta_n(Y_1,\cdots,\hat Y_j,\cdots\hat Y_i,\cdots,Y_{n+1},Y_i)$,

\noindent for $1\leq i<j\leq n+1$, we have $\Delta_n(Y_1,\cdots,Y_i,\cdots\hat Y_j,\cdots,Y_{n+1})=(-1)^{n-i}
\Delta_n(Y_1,\cdots\hat Y_i,\cdots\hat Y_j,\cdots,Y_{n+1},Y_i)$,

\noindent and for $1\leq i\leq n+1$, we have $\Delta_{n+1}(Y_1,\cdots\hat Y_i,\cdots,Y_{n+1},Y_i)=(-1)^{n+1-i}\Delta_{n+1}(\underline Y)$.

\ms
Then it follows with \eqref{F15} that for $1\leq i\leq n+1$, 
$\Delta_n(\Delta_n(\underline{\hat Y_1}),\Delta_n(\underline{\hat Y_2}),\cdots,\hat\Delta_n(\underline{\hat Y_i}),\cdots,\Delta_n(\underline{\hat Y_{n+1}}))=\Delta_n(\Delta(\hat Y_1,Y_2,\cdots,Y_i,\cdots,Y_{n+1}),\cdots,\Delta_n( Y_1,\cdots,\hat Y_{i-1},Y_i,\cdots,Y_{n+1}),\Delta_n( Y_1,\cdots,Y_{i-1},Y_i,\hat Y_{i+1},\cdots,Y_{n+1}),\cdots,$

\noindent $\Delta_n( Y_1,\cdots,Y_i,\cdots,\hat Y_{n+1}))=
(-1)^{(n+1-i)(i-1)+(n-i)(n+1-i)}(-1)^{\left\lfloor\frac{n}{2}\right\rfloor}\tilde\Delta_{n,1}(\underline{\hat Y_i},Y_i)=$

\noindent
$(-1)^{(n+1-i)(i-1)+(n-i)(n+1-i)}(-1)^{\left\lfloor\frac{n}{2}\right\rfloor}(-1)^{(n-1)(n-i+1)}Y_i^{q^{n-1}}\Delta_{n+1}(\underline{Y})^{1+q+\cdots+q^{n-2}}$
\noindent \begin{eqnarray}
\qquad\qquad\qquad\qquad\qquad\qquad\qquad\qquad\qquad\qquad\qquad\qquad =(-1)^{\left\lfloor\frac{n}{2}\right\rfloor}Y_i^{q^{n-1}}\Delta_{n+1}(\underline{Y})^{1+q+\cdots+q^{n-2}}.\label{NF16}\end{eqnarray}
%%%%%%%%%%%%%%%%%%%%%%%
Thus, by developing the determinant $\Delta_{n+1}(\Delta_n(\underline{\hat Y_1}),\cdots,\Delta_n(\underline{\hat Y_i}),\cdots,\Delta_n(\underline{\hat Y_{n+1}}))$ along the first row,  it follows that

\noindent $\Delta_{n+1}(\Delta_n(\underline{\hat Y_1}),\cdots,\Delta_n(\underline{\hat Y_i}),\cdots,\Delta_n(\underline{\hat Y_{n+1}})) =\Delta_n(\underline{\hat Y_1})\Delta_n(\hat\Delta_n(\underline{\hat Y_1}),\Delta_n(\underline{\hat Y_2}),\cdots,\Delta_n(\underline{\hat Y_i}),
\cdots,\Delta_n(\underline{\hat Y_{n+1}}))^q$

\noindent $\qquad \qquad\qquad\qquad\qquad\qquad\qquad\qquad-\Delta_n(\underline{\hat Y_2})\Delta_n(\Delta_n(\underline{\hat Y_1}),\hat\Delta_n(\underline{\hat Y_2}),\cdots,\Delta_n(\underline{\hat Y_i}),
\cdots,\Delta_n(\underline{\hat Y_{n+1}}))^q+\cdots+$  

\noindent $\qquad\qquad\qquad\qquad\qquad\qquad\qquad(-1)^{n+2}\Delta_n(\underline{\hat Y_{n+1}})\Delta_n(\Delta_n(\underline{\hat Y_1}),\cdots,\Delta_n(\underline{\hat Y_i}),
\cdots,\Delta_n(\underline{\hat Y_{n}}),\hat\Delta_n(\underline{\hat Y_{n+1}}))^q=$

\noindent
$\qquad\qquad\qquad\qquad\qquad\qquad\qquad\qquad\ (-1)^{\left\lfloor\frac{n}{2}\right\rfloor}(\sum_{1\leq i\leq n+1}(-1)^{i+1}\Delta_{n}(\underline{\hat Y_i})Y_i^{q^n})\Delta_{n+1}(\underline{Y})^{q(1+q+\cdots+q^{n-2})}=$

\noindent
$(-1)^{\left\lfloor\frac{n}{2}\right\rfloor}(-1)^n(\sum_{1\leq i\leq n+1}(-1)^{i+n+1}Y_i^{q^n}\Delta_{n}(\underline{\hat Y_i}))\Delta_{n+1}(\underline{Y})^{q(1+q+\cdots+q^{n-2})}=(-1)^{\left\lfloor\frac{n+1}{2}\right\rfloor}\Delta_{n+1}(\underline{Y})^{1+q+\cdots+q^{n-1}}.$ 

%%%%%%%%%%%%%%%%%%%%%%%%

This is \eqref{F13} for 
$m=0$ and $n+1$. This also shows the equality of the coefficients of higher 
degree in \eqref{F13} for $m=1$ and $n+1$.

\ms
{\bf B2.} We compare the zeros with multiplicity in the two members of \eqref{F13} for $m=1$ and $n+1$.

We write

\noindent $G:=\Delta_{n+1}(\Delta_{n+1}(\underline{\hat Y_1},X),\cdots,(-1)^{i+1}\Delta_{n+1}(\underline{\hat Y_i},X),\cdots,(-1)^{n+2}\Delta_{n+1}(\underline{\hat Y_{n+1}},X)) $ and 

\noindent $D:= X^{q^{n}}\Delta_{n+2}(\underline{Y},X)^{1+q+\cdots+q^{n-1}}.$

We are first interested in $X=0$, for that we notice that 

$\frac{G}{X^{1+q+\cdots+q^n}}=\Delta_{n+1}(\frac{\Delta_{n+1}(\underline{\hat Y_1},X)}{X},\cdots,(-1)^{i+1}\frac{\Delta_{n+1}(\underline{\hat Y_i},X)}{X},\cdots,(-1)^{n+2}\frac{\Delta_{n+1}(\underline{\hat Y_{n+1}},X)}{X}) $
whose constant term is
$(-1)^{n(n+1)}\Delta_{n+1}(\Delta_n(\underline{\hat Y_1}),\cdots,(-1)^{i+1}\Delta_n(\underline{\hat Y_i}),\cdots,(-1)^{n+2}\Delta_n(\underline{\hat Y_{n+1}}))^q$.
On the other hand

\noindent $\frac{D}{X^{1+q+\cdots+q^n}}=(\frac{\Delta_{n+2}(\underline{Y},X)}{X})^{1+q+\cdots+q^{n-1}}$ whose constant term is
$(-1)^{(n+1)n}\Delta_{n+1}(\underline{Y})^{q(1+q+\cdots+q^{n-1})}$. 

Then we have equality and non nullity of constant terms by {\bf B1.}, which ensures in particular that the multiplicity of $X=0$ is 
$1+q+\cdots+q^n$ in $G$ and in $D$.

Thanks to (\ref{F2}), we can handle the other zeros. Let $\underline{\epsilon}:=(\epsilon_1,\epsilon_2,\cdots,\epsilon_{n+1})\in \F_q^{n+1}-(0,\cdots,0)$, and 
$x_{\underline{\epsilon}}:=\sum_{1\leq j\leq n+1}\epsilon_jY_j$. 
We need to show that $x_{\underline{\epsilon}}$ is a root of  $G$ with multiplicity
$1+q+\cdots+q^{n-1}$. 

With \eqref{F1} we get

$$G=\prod_{1\leq i\leq n}\prod_{\alpha_{i-1}\in \F_q}\cdots\prod_{\alpha_1\in \F_q} (\Delta[i](\underline Y,X)+\alpha_{i-1}\Delta[i-1](\underline Y,X)+\cdots+\alpha_{1}\Delta[1](\underline Y,X))$$
where $\Delta[i](\underline Y,X):=(-1)^{i+1}\Delta_{n+1}(\underline{\hat Y_i},X)$ and so with Proposition \ref{prop2}
 
 $$G=\prod_{1\leq i\leq n}\prod_{\alpha_{i-1}\in \F_q}\cdots\prod_{\alpha_{1}\in \F_q}\Delta_{(\alpha_1,\cdots,\alpha_{i-1},1,0,\cdots,0)}(\underline Y,X).$$
 \noindent  where $\Delta_{(\alpha_1,\cdots,\alpha_{i-1},1,0,\cdots,0)} =\Delta_{\varphi_i}$ with $\varphi_i=\alpha_1Y_1^\star+\cdots+\alpha_{i-1}Y_i^\star+Y_i^\star $ and $(Y_i^\star)_{1\leq i\leq n}$ is the dual basis of $(Y_i)_{1\leq i\leq n}$. 
The roots of $\Delta_{(\alpha_1,\cdots,\alpha_{i-1},1,0,\cdots,0)}(\underline Y,X)$ seen as a polynomial in  $X$ and coefficients in $\F_q(\underline Y)$ are
simple  (Proposition \ref{prop2}) and 
$\Delta_{(\alpha_1,\cdots,\alpha_{i-1},1,0,\cdots,0)}(x_{\underline{\epsilon}})=0$
if and only if
 $\epsilon_i+\alpha_{i-1}\epsilon_{i-1}+\cdots+\alpha_{1}\epsilon_{1}=0$. Thus the multiplicity of $x_{\underline{\epsilon}}$ in $G$ is equal to the cardinality of the $ (\alpha_1:\alpha_2:\cdots:\alpha_{n+1}) \in \mathbb{P}^n(\F_q)$ which belong to the hyperplane  $\sum_{1\leq i\leq n+1}\epsilon_i\alpha_i=0$ i.e. 
$1+q+\cdots+q^{n-1}$.  
Hence we get \eqref{F13} for $m=1$ and $n+1$.

\subsection{Second proof of Theorem \ref{thm1} by a matrix interpretation in the case $m=0$ }

The following theorem is of interest independently of the rest. It gives indeed a relation between a generic Moore matrix and the Moore matrix of the cofactors of its first row; relation analogous to the classical relation between a square matrix and its comatrix. The $m=0$ case of Theorem \ref{thm1}  is then an immediate corollary by taking the determinants. 
\newpage
\begin{thm}\label{inverse}
 Let  $Y_1,Y_2,\cdots,Y_n$, be $n$ indeterminates over $\F_q$ and
 
${\cal{M}}_{n}(\Delta_n(\underline{\hat Y_1}),.,(-1)^{i-1}\Delta_n(\underline{\hat Y_i}),
\cdots,(-1)^{n-1}\Delta_n(\underline{\hat Y_{n}}))$, the Moore matrix of the cofactors 

\noindent $(\Delta_n(\underline{\hat Y_1}),.,(-1)^{i-1}\Delta_n(\underline{\hat Y_i}),
\cdots,(-1)^{n-1}\Delta_n(\underline{\hat Y_{n}}))$ of the first row of ${\cal{M}}_n(\underline{Y})$.
Then one gets

$$
 {\cal{M}}_{n}(\Delta_n(\underline{\hat Y_1}),\cdots,(-1)^{i-1}\Delta_n(\underline{\hat Y_i}),
\cdots,(-1)^{n-1}\Delta_n(\underline{\hat Y_{n}}))\ ^t{\cal{M}}_n(\underline Y)=
$$
\begin{eqnarray}
=\left(
\begin{array}{ccccccc}
0&.&.&\cdots&.&0& (-1)^{n-1}\Delta_n(\underline Y)  \\
\Delta_n(\underline Y)&0&.&\cdots&.&0&0\\
\alpha_1&\Delta_n(\underline Y)^q&0&\cdots&.&0&0  \\
\alpha_2&\alpha_1^q&\Delta_n(\underline Y)^{q^2}&\cdots&.&0&0\\
\vdots&\vdots&\vdots&\cdots&.&\vdots&\vdots\\
\alpha_{n-2}&\alpha_{n-3}^q&.&\cdots&\alpha_1^{q^{n-3}}&\Delta_n(\underline Y)^{q^{n-2}}&0

\end{array}
\right)\label{FF16}
\end{eqnarray}
where $\alpha_k:=
\Delta_n(\underline{\hat Y_1})^{q^{k+1}}Y_1+\cdots+(-1)^{i-1}\Delta_n(\underline{\hat Y_i})^{q^{k+1}}Y_i+\cdots+(-1)^{n-1}\Delta_n(\underline{\hat Y_n})^{q^{k+1}}Y_n $.
\end{thm}

\begin{proof}
We write $
 {\cal{M}}_{n}(\Delta_n(\underline{\hat Y_1}),\cdots,(-1)^{i-1}\Delta_n(\underline{\hat Y_i}),
\cdots,(-1)^{n-1}\Delta_n(\underline{\hat Y_{n}}))\ ^t{\cal{M}}_n(\underline Y)=:[m_{i,j}]_{1\leq i,j\leq n}.
$

 Since  $(-1)^{i-1}\Delta_n(\underline{\hat Y_i})^q$ is the  cofactor of  $Y_i$ in the Moore matrix ${\cal{M}}_{n}(\underline{Y})$, we get the following formulas 
 \begin{eqnarray}
 \Delta_n(\underline{\hat Y_1})^qY_1+\cdots+(-1)^{i-1}\Delta_n(\underline{\hat Y_i})^qY_i+\cdots+(-1)^{n-1}\Delta_n(\underline{\hat Y_n})^qY_n=\Delta_n(\underline Y)\label{FF17}
 \end{eqnarray}
 and for $1\leq k\leq n-1,$
 \begin{eqnarray}
 \Delta_n(\underline{\hat Y_1})^qY_1^{q^k}+\cdots+(-1)^{i-1}\Delta_n(\underline{\hat Y_i})^qY_i^{q^k}+\cdots+(-1)^{n-1}\Delta_n(\underline{\hat Y_n})^qY_n^{q^k}=0.\label{FF18}
 \end{eqnarray}
 
Since  $(-1)^{i-1}\Delta_n(\underline{\hat Y_i})$ is the  cofactor of $Y_i^{q^{n-1}}$, we get the following formulas 

 \begin{eqnarray}
\Delta_n(\underline{\hat Y_1})Y_1^{q^{n-1}}+\cdots+(-1)^{i-1}\Delta_n(\underline{\hat Y_i})Y_i^{q^{n-1}}+\cdots+(-1)^{n-1}\Delta_n(\underline{\hat Y_n})Y_n^{q^{n-1}}=(-1)^{n-1}\Delta_n(\underline Y)
  \end{eqnarray}\label{FF19}
\noindent and for $0\leq k\leq n-2,$
 \begin{eqnarray}
 \Delta_n(\underline{\hat Y_1})Y_1^{q^k}+\cdots+(-1)^{i-1}\Delta_n(\underline{\hat Y_i})Y_i^{q^k}+\cdots+(-1)^{n-1}\Delta_n(\underline{\hat Y_n})Y_n^{q^k}=0.\label{FF20}
 \end{eqnarray}
It follows from the relations \eqref{FF19} and \eqref{FF20} that $m_{1,j}=0$ for 
$1\leq j\leq n_1$ and that $m_{1,n}=(-1)^{n-1}\Delta_n(\underline Y).$

Let now $2\leq i\leq n$.
Raising \eqref{FF17} to the power $q^{i-1}$, it follows that 
$m_{i,i-1}=\Delta_n(\underline Y)^{q^{i-1}}.$
Raising \eqref{FF18} to the power $q^{i-1}$, it follows that 
$m_{i,j}=0$ for $i\leq j\leq n$.

In conclusion the matrix $[m_{i,j}]_{1\leq i,j\leq n}$ satisfies \eqref{FF16}.
\end{proof}

By taking the determinant of the matrices in \eqref{FF16} we obtain that $\gamma=1$ in the case $m=0$ of Corollary \ref{coro2.3}, 
Theorem \ref{thm1} follows.

\subsection{A matrix interpretation of the general case $(n,m)$.}

The following theorem is a generalization of Theorem \ref{thm2} adapted to a matrix interpretation of the general case of Theorem \ref{thm1}. 
Theorem \ref{thm2} corresponds to the case $m=0$ i.e. $\underline X=\emptyset.$
%\newpage 
\begin{thm}\label{thm2} For $n\geq 2$, $m\geq 1$ let $Y_1,Y_2,\cdots,Y_n,X_1,X_2,\cdots,X_m$ be $n+m$ indeterminates over $\F_q$, $\delta_i:=(-1)^{i-1}\Delta_{n+m-1}((\underline{\hat Y_i}),(\underline{X}))$ for $1\leq i\leq n$, $ \delta_i:=(-1)^{i-1}\Delta_{n+m-1}((\underline{Y}),(\underline{\hat X_i}))$ for $n+1\leq i\leq n+m$. 

Let  $A:=[a_{i,j}]_{1\leq i,j\leq n+m}$, where $a_{i,j}=(\delta_j)^{q^{i-1}}$ for $1\leq i\leq n,\ 1\leq j\leq n+m$ and $a_{i,i-n}=1$ for $n+1\leq i\leq n+m$ and $a_{i,j}=0$ for $n+1\leq i\leq n+m$ and  $j\neq n-i$.

Hence
\begin{eqnarray}
 A= \left(
\begin{array}{cc}
{\cal{M}}_{n}(\delta_1,\delta_2,\cdots,\delta_n) &{\cal{M}}_{n,m}(\delta_{1+n},\delta_{2+n},\cdots,\delta_{n+m})\\
 0\in M_{m,n}(\F_q[Y,X])&Id_m
\end{array}
\right).
\label{F16}\end{eqnarray}

Then
\begin{eqnarray}
A\ ^t{\cal{M}}_{n+m}(\underline Y,\underline X)=[m_{i,j}]_{1\leq i,j\leq n}=:M,\ {\rm with}
\label{F17}\end{eqnarray}

\noindent  $m_{1,j}=0\  {\rm for}\  1\leq j\leq n+m-1,\  m_{1,n+m}=(-1)^{n+m-1}\Delta_{n+m}(\underline Y,\underline X),\ m_{2,1}=\Delta_{n+m}(\underline Y,\underline X), \ m_{2,j}=0\ {\rm for}\ 2\leq j\leq n+m\ {\rm and \  si} \ 3\leq i\leq n,\ 1\leq j\leq i-2,\ $
\begin{eqnarray}
m_{i,j}=\alpha_{i-j-1}^{q^{j-1}},{\rm with\ } 
\alpha_k:=
\delta_1^{q^{k+1}}Y_1+\cdots+\delta_n^{q^{k+1}}Y_n+\delta_{n+1}^{q^{k+1}}X_1+\cdots+\delta_{n+m}^{q^{k+1}}X_m,\label{F18}
\end{eqnarray}
$ m_{i,i-1}=\Delta_{n+m}(\underline Y,\underline X)^{q^i}\ and \ m_{i,j}=0\ {\rm for}\ i\leq j\leq n+m.$
 
 In matrix writing we have
 $M=\left(
\begin{array}{cc}
 M_1&M_2\\
 M_3&M_4
\end{array}
\right)$, where 

\noindent $M_1:=\left(
\begin{array}{ccccccc}
0&.&.&\cdots&.&0& 0  \\
\Delta_{n+m}(\underline Y,\underline X)&0&.&\cdots&.&0&0\\
\alpha_1&\Delta_{n+m}(\underline Y,\underline X)^q&0&\cdots&.&0&0  \\
\alpha_2&\alpha_1^q&\Delta_{n+m}(\underline Y,\underline X)^{q^2}&\cdots&.&0&0\\
\vdots&\vdots&\vdots&\cdots&.&\vdots&\vdots\\
\alpha_{n-2}&\alpha_{n-3}^q&.&\cdots&\alpha_1^{q^{n-3}}&\Delta_{n+m}(\underline Y,\underline X)^{q^{n-2}}&0

\end{array}
\right),$

\noindent $M_2:=\left(
\begin{array}{ccccc}
 0&0&.&.&(-1)^{m+n-1}\Delta_{n+m}(\underline Y,\underline X)\\
 0&0&.&.&0\\
 .&.&.&.& .\\
 0&0&.&.&0
\end{array}
\right)$, $M_3=\ ^t{\cal{M}}_{n,m}(X_1,X_2,\cdots,X_m),$

\noindent $M_4=\ ^t{\cal{M}}_{m}(X_1^{q^n},X_2^{q^n},\cdots,X_m^{q^n}).$

\end{thm}

\msn {\sl Proof of Theorem \ref{thm2}}
 
 \ms We can consider $\delta_i^q$ as the cofactor of $Y_i$ or $X_i$, in the Moore matrix ${\cal{M}}_n(\underline{Y},\underline{X})$, so we have the following formulas
 \begin{eqnarray}
  \delta_1^qY_1+\delta_2^qY_2+\cdots+\delta_n^qY_n+ \delta_{n+1}^qX_1+\delta_{n+2}^qX_2+\cdots+\delta_{n+m}^qX_m=\Delta_{n+m}(\underline Y,\underline X)\label{F24}
 \end{eqnarray}
 \begin{eqnarray}
  \delta_1^qY_1^{q^k}+\delta_2^qY_2^{q^k}+\cdots+\delta_n^qY_n^{q^k}+ \delta_{n+1}^qX_1^{q^k}+\delta_{n+2}^qX_2^{q^k}+\cdots+\delta_{n+m}^qX_m^{q^k}=0\label{F25}
 \end{eqnarray}
 for $1\leq k\leq n+m-1$.
 
 \ms We can also consider $\delta_i$ as the cofactor of $Y_i^{q^{n+m-1}}$ or of $X_i^{q^{n+m-1}}$ in the Moore matrix ${\cal{M}}_n(\underline{Y},\underline{X})$. We thus have the following formulas 

 \msn$ \delta_1Y_1^{q^{n+m-1}}+\delta_2Y_2^{q^{n+m-1}}+\cdots+\delta_nY_n^{q^{n+m-1}}+ \delta_{n+1}X_1^{q^{n+m-1}}+\cdots+\delta_{n+m}X_m^{q^{n+m-1}}=$
  \begin{eqnarray} =(-1)^{n+m-1}\Delta_{n+m}(\underline Y,\underline X)
 \label{F26}\end{eqnarray}
 \begin{eqnarray}
  \delta_1Y_1^{q^k}+\delta_2Y_2^{q^k}+\cdots+\delta_nY_n^{q^k}+ \delta_{n+1}X_1^{q^k}+\delta_{n+2}X_2^{q^k}+\cdots+\delta_{n+m}X_m^{q^k}=0
 \label{F27}\end{eqnarray}
 for $0\leq k\leq n+m-2$.

\ms It follows from the relations ({\ref{F16}}) and ({\ref{F17}}) that the first line of $A\ ^t{\cal{M}}_n(\underline{Y},\underline{X})$ is the same as the first line of $M=[m_{i,j}]_{1\leq i,j\leq n}.$

Then to show the equality between the lines of index $i$ with $2\leq i\leq n$, it is enough to raise the relations ({\ref{F24}}) and ({\ref{F25}}) to the power $q^{i-1}$ and to use the definition of $\alpha_k$ for $1\leq k\leq n-2$.

The equality between the lines of index $i$ with $n+1\leq i\leq n+m$ is immediate. All this shows the relation ({\ref{F17}}).

\begin{coro} Theorem \ref{thm1} is a consequence of the matrix equality in Theorem \ref{thm2}.
 
\end{coro}

\begin{proof}
 Expanding the determinant of $M$ according to the first line, we have
 \begin{eqnarray}
  \det M=\Delta_{n+m}(\underline Y,\underline X)\det N
  \ {\rm with}\ N=\left(
\begin{array}{cc}
 N_1&N_2\\
 N_3&N_4
\end{array}
\right)
 \label{F19} \end{eqnarray}
  \noindent $N_1=\left(
\begin{array}{cccccc}
\Delta_{n+m}(\underline Y,\underline X)&0&.&\cdots&.&0\\
\alpha_1&\Delta_{n+m}(\underline Y,\underline X)^q&0&\cdots&.&0  \\
\alpha_2&\alpha_1^q&\Delta_{n+m}(\underline Y,\underline X)^{q^2}&\cdots&.&0\\
\vdots&\vdots&\vdots&\cdots&.&\vdots\\
\alpha_{n-2}&\alpha_{n-3}^q&.&\cdots&\alpha_1^{q^{n-3}}&\Delta_{n+m}(\underline Y,\underline X)^{q^{n-2}}

\end{array}
\right),$

\msn
$N_2$ is the zero matrix in $M_{n-1,m-1}(\F_q[\underline Y,\underline X]),\ N_3=\ ^t{\cal{M}}_{n-1,m}(X_1,X_2,\cdots,X_m),$

\noindent 
$N_4=\ ^t{\cal{M}}_{m}(X_1^{q^{n-1}},X_2^{q^{n-1}},\cdots,X_m^{q^{n-1}}).$
 
It is then clear that  
\begin{eqnarray}
 \det N=\Delta_{n+m}(\underline Y,\underline X)^{1+q+\cdots+q^{n-2}}\Delta_{m}(\underline X)^{q^{n-1}}.
\label{F20}\end{eqnarray}

Thus with (\ref{F19}) and  (\ref{F20}), one gets
\begin{eqnarray}
 \det M=\Delta_{n+m}(\underline Y,\underline X)\Delta_{n+m}(\underline Y,\underline X)^{1+q+\cdots+q^{n-2}}\Delta_{m}(\underline X)^{q^{n-1}}.
\label{F21}\end{eqnarray}

It follows from (\ref{F17}), that $\det M=\Delta_{n+m}(\underline Y,\underline X)\det A$ and from (\ref{F6}) that $\det A=\Delta_{n}(\delta_1,\delta_2,\cdots,\delta_n)$, thus 
\begin{eqnarray}
 \det M=\Delta_{n+m}(\underline Y,\underline X)\Delta_{n}(\delta_1,\delta_2,\cdots,\delta_n).\label{F22}
\end{eqnarray}

Since (cf. Proposition \ref{prop2.1}),
$\Delta_{n+m}(\underline Y,\underline X)\neq 0$, and that $\F_q[(\underline Y,\underline X)$ is an integral  ring, the equality (\ref{F13}) in Theorem \ref{thm1} for $m\geq 1$. follows from  (\ref{F21}) and  (\ref{F22}).

 Finally, the equality of the coefficients of highest degree in $X_1$ in
 formula (\ref{F13}) in Theorem \ref{thm1} for $m=1$ gives, as it is noticed in the first proof,
formula (\ref{F13}) in Theorem \ref{thm1} for $m=0$.
\end{proof}

\section{Two illustrations of M-O.Ore determinants}

\subsection{The map $(a_1,\cdots,a_n)\in K^n\to (\Delta_{n-1}(\underline{\hat{a_i}}))_{1\leq i\leq n}\in K^n$}
\begin{prop}\label{prop5.1}
Let $K$ be an algebraically closed field with characteristic  $p>0$. Let us denote by $V(\Delta_n):=\{(a_1,a_2,\cdots,a_n):=\underline a\in K^n\ |\  
 \Delta_n(\underline a)= 0\}$. 
The map $\varphi: \underline a:=(a_1,a_2,\cdots,a_n)\in K^n\to (\Delta_{n-1}(\underline{\hat{a_i}}))_{1\leq i\leq n}\in K^n$ induces an onto map from  $K^n-V(\Delta_n)$ to itself. Moreover for $\underline a$ and $\underline{a^\prime}$ in 
 $K^n-V(\Delta_n)$, one has $\varphi(\underline a)=\varphi(\underline {a^\prime})$
if and only if $\underline{a^\prime}=\lambda \underline{a}$ where 
$\lambda^{1+q+\cdots+q^{n-2}}=1$. 
\end{prop}

\begin{proof} Let $(a_1,a_2,\cdots,a_n)\in K^n-V(\Delta_n)$ and $b_i:=\Delta_{n-1}(\underline{\hat{a_i}})$ for $1\leq i\leq n$. Since $\Delta_n(\underline b)=$

\noindent $(-1)^{\left\lfloor\frac{n}{2}\right\rfloor}\Delta_n(\underline a)^{1+q+\cdots+q^{n-2}}$ (we recognize \eqref{F13} for $m=0$), it follows that $\varphi(K^n-V(\Delta_n))\subset K^n-V(\Delta_n)$. Then
$\varphi^2(\underline a)=(\Delta_{n-1}(\underline{\hat{b_i}}))_{1\leq i\leq n}=(\Delta_{n-1}(\Delta_{n-1}(\underline{\hat{a_1}}),\cdots,\Delta_{n-1}(\underline{\hat{a_{i-1}}}),\Delta_{n-1}(\underline{\hat{a_{i+1}}}),\cdots,\Delta_{n-1}(\underline{\hat{a_n}})))_{1\leq i\leq n}=$

\noindent
$=((-1)^{\left\lfloor\frac{n-1}{2}\right\rfloor}\Delta_n(\underline{a})^{1+q+\cdots+q^{n-3}}a_i^{q^{n-2}})_{1\leq i\leq n}$ we recognize there the equality \eqref{NF16} at rank $n$ obtained in the first proof of Theorem \ref{thm1}. It follows that 
$$\varphi^2(\underline a)=(-1)^{\left\lfloor\frac{n-1}{2}\right\rfloor}\Delta_n(\underline{a})^{1+q+\cdots+q^{n-3}}(\underline a)^{q^{n-2}}.$$

Let $\lambda\in K-\{0\}$, then $$\varphi^2(\underline{\lambda a})=\lambda^{(1+q+\cdots+q^{n-1})(1+q+\cdots+q^{n-3})+q^{n-2}}\varphi^2(\underline{a}).$$ 
Note that we have the equality $(1+q+\cdots+q^{n-1})(1+q+\cdots+q^{n-3})+q^{n-2}=(1+q+\cdots+q^{n-2})^2$ according to the fact that $\varphi(\underline{\lambda a})=\lambda^{1+q+\cdots+q^{n-2}}\varphi(\underline{a})$. Thus by taking 
$\lambda$ with $\lambda^{(1+q+\cdots+q^{n-1})^2}=\Delta_n(\underline{a})^{-(1+q+\cdots+q^{n-3})}$, one gets $\varphi^2(\underline{\lambda a})=(\underline{a})^{q^{n-2}}$. Hence the surjectivity of $\varphi^2$ and therefore of $\varphi$. 

Let us now examine the injectivity defect of the map $\varphi$.

Let $\underline a$ and $\underline{a^\prime}$ be 
in $K^n-V(\Delta_n)$ such that $\varphi(\underline a)=\varphi(\underline {a^\prime})$, then $\varphi^2(\underline a)=\varphi^2(\underline {a^\prime})$ and so
$\Delta_n(\underline{a})^{1+q+\cdots+q^{n-3}}a_i^{q^{n-2}}$

\noindent
$=\Delta_n(\underline{a^\prime})^{1+q+\cdots+q^{n-3}}{a^\prime_i}^{q^{n-2}}$. Thus there is $\lambda \in K$ such that $\underline{a^\prime}=\lambda \underline{a}$, hence  $\lambda^{1+q+\cdots+q^{n-2}}\varphi(\underline{a})=\varphi(\underline{a})$ and so $\lambda^{1+q+\cdots+q^{n-2}}=1$. The converse is immediate.
 
\end{proof}

\begin{rem} Proposition \ref{prop5.1} works the same if we replace the map $\varphi $ by the map $\varphi_1 $ where
$\varphi_1: \underline a:=(a_1,a_2,\cdots,a_n)\in K^n\to ((-1)^{i-1}\Delta_{n-1}(\underline{\hat{a_i}}))_{1\leq i\leq n}\in K^n$ as $\varphi^2=\varphi_1^2$.
\end{rem}

\subsection{On $K$-etale algebras and Elkies pairing}

In this paragraph, unless expressly mentionned $K$ is a field of characteristic $p>0$, $K^{alg}$ is an algebraic closure of $K$ and $F$ is the Frobenius automorphism defined by $F(x)=x^p$ for $x\in K^{alg}$. 

Let $\underline f:=(f_1,f_2,\cdots,f_n)\in K^n$ with $\Delta_n(\underline f)\neq 0$, i.e. $f_1,f_2,\cdots,f_n$ are $\F_p$ free. 
We intend to study the $K$-algebra $A:=\frac{K[W_i,1\leq i\leq n]}{(W_i^p-W_i-f_i)_{ 1\leq i\leq n}}$, in particular its group of $K$-automorphisms $\Aut_K A$, and to exhibit a special generator of the $K$-algebra $A$ and a subgroup $(\Z/p\Z)^n\subset \Aut_K A $ whose action on $A$ is dictated by an associated  Elkies pairing (Section 3.2.A.).

\begin{prop} \label{prop5.2}
Let $n\geq 1$ and $\underline f:=(f_1,f_2,\cdots,f_n)\in K^n$  where
$\Delta_n(\underline f)\neq 0$.  Let $V$ be the $\F_p$-vector space  $\frac{(\sum_{1\leq i\leq n}\F_pf_i)+(F-Id)(K)}{(F-Id)(K)}$ of dimension  $r\leq n$ and $I\sqcup J$ be a partition of  $\{1,2,\cdots,n\}$ such that $f_{i},\ i\in I$ induces an $\F_p$-basis of the vector space $V$.
Let $A$ be the $K$-algebra $\frac{K[W_k,1\leq k\leq n]}{(P_k)_{ 1\leq k\leq n}}$ where $P_k:=W_k^p-W_k-f_k$, then
$A$ is an etale $K$-algebra isomorphic to $L^{p^{n-r}}$, the cartesian product of $p^{n-r}$ copies of $L$,  where
$L\subset K^{alg}$ is a field which is a Galois extension of $K$ of group $(\Z/p\Z)^r$ and $L\simeq \frac{K[W_{k},\ k\in I]}{(P_k)_{k\in I}}$.

The group of $K$-automorphisms $\Aut_K A$ is then 
isomorphic to a semidirect product of the groups $\mathfrak S_{p^{n-r}}$ and $((\Z/p\Z)^r)^{p^{n-r}}$.

Moreover if $w_i$ denotes the canonical image of $W_i$ in $A$ and if
 
 \begin{eqnarray}
w:= \delta_{\underline{w}}(\underline{f}):=\left|
\begin{array}{cccc}
w_1 &w_2&\cdots&w_n\\
f_1&f_2&\cdots&f_n  \\
f_1^{p}&f_2^{p}&\cdots&f_n^p   \\
\vdots&\vdots&\cdots&\vdots\\   
f_1^{p^{n-2}}&f_2^{p^{n-2}}&\cdots&f_n^{p^{n-2}} 
\end{array}
\right|\in K[w_1,w_2,...,w_n], \label{F28}
\end{eqnarray} then

\noindent $w_i=\frac{\Delta_n(\Delta_{n-1}(\underline{\hat{f}_1}),\cdots,(-1)^{i-2}\Delta_{n-1}(\underline{\hat{f}_{i-1}}),w,(-1)^{i}\Delta_{n-1}(\underline{\hat{f}_{i+1}}),\cdots, (-1)^{n-1}\Delta_{n-1}(\underline{\hat{f}_n}) )}{\Delta_n(\Delta_{n-1}(\underline{\hat{f}_1}),\cdots,(-1)^{j-1}\Delta_{n-1}(\underline{\hat{f}_{j}}),\cdots, (-1)^{n-1}\Delta_{n-1}(\underline{\hat{f}_n}) )}-(f_i+f_i^p+...+f_i^{p^{n-2}})\in K[w]
$, where  $\Delta_{n-1}(\underline{\hat{f}_i})=\Delta_{n-1}(f_1,f_2,\cdots, f_{i-1},f_{i+1},\cdots, f_{n})$ and 

\begin{eqnarray} A=K[w_1,w_2,\cdots,w_n]=K[w]\simeq \frac{K[W]}{(Q(W))}\label{F29}
\end{eqnarray}
\noindent where 

\noindent $Q(W)=\frac{\Delta_{n+1}(\Delta_{n-1}(\underline{\hat{f}_1}),\cdots,\Delta_{n-1}(\underline{\hat{f}_n}),W)}{\Delta_{n}(\Delta_{n-1}(\underline{\hat{f}_1}),\cdots,\Delta_{n-1}(\underline{\hat{f}_n}))  }-\Delta_n(\underline f)^{p^{n-1}}=$
 
\noindent $W^{p^n}+(\sum_{1\leq i\leq n-1}(-1)^{n-i}\Delta_n(\underline f)^{p^{n-1}-p^{i-1}-p^i}(\Delta[n-i](\underline f))^{p^{i-1}}W^{p^{n-i}}) +(-1)^n\Delta_n(\underline f)^{p^{n-1}-1}W-\Delta_n(\underline f)^{p^{n-1}},$

\noindent $\Delta[i](\underline f):= \det (\underline f,F(\underline f),\cdots,\hat F^i(\underline f),\cdots,F^{n}(\underline f))$ 
and  $Q(w)=0$. 

\end{prop}

\begin{proof}
 i) Let us show that $A$ is isomorphic to the $K$-algebra $L^{p^{n-r}}.$

By definition of $I$, we have $V=\frac{(\sum_{i\in I}\F_pf_i)+(F-Id)(K)}{(F-Id)(K)}$. 
The Artin-Schreier theory 
(\cite{B} chap.5 p.88 \S .11 Theorem 5) 
says that
 $(F-Id)^{-1}(\sum_{i\in I}\F_pf_i+(F-Id)(K))\subset K^{alg}$ is a Galois extension 
$L/K$ of group $Hom(V,\F_p)\simeq (\Z/p\Z)^r$ and that 
\begin{eqnarray}
L=\bigoplus_{0\leq \alpha_i<p,\ i\in I} K \prod_{i\in I}x_i^{\alpha_i}\ {\rm where}\ x_i\in K^{alg}\ {\rm and}\  P_i(x_i)=0. \label{38}
\end{eqnarray}
Let $\pi$ be the $K$-algebra homomorphism of $K[W_i,\ i\in I]$ onto $L$ mapping $\pi(W_i)$ to $x_i$. It follows from \eqref{38} that $\pi$   induces a $K$-algebra homomorphism $\pi': \frac{K[W_{i},\ i\in I]}{(P_i)_{i\in I}}=K[w_i,\ i\in I]\to L$ which is surjective and as $K[w_i,\ i\in I]=\sum_{0\leq \alpha_i<p,\ i\in I} K \prod_{i\in I}w_i^{\alpha_i}$ we get a $K$-algebra isomorphism 
\begin{eqnarray}
\frac{K[W_{i},\ i\in I]}{(P_i)_{i\in I}}\simeq L.\label{39}
\end{eqnarray}

On the other hand  we have for $j\in J$, $f_j=\sum_{i\in I}\lambda_{j,i}f_{i}+g_j^p-g_j$
with $\lambda_{j,i}\in \F_p$ and $g_j\in K$. Thus for $j\in J$, and if $W'_j:=W_j+\sum_{i\in I}\lambda_{j,i}W_{i}$, one gets 
$K[W_k,1\leq k\leq n]= K[W_{i},\ i\in I, W'_j,\ j\in J ]$ and for 
$j\in J$ one has 
${W'_j}^p-W'_j-(g_j^p-g_j)={W_j}^p-W_j-f_j+\sum_{i\in I}\lambda_{j,i}({W_{i}}^p-W_{i}-f_i)$ and so if $P'_j(W'_j):={W'_j}^p-W'_j-(g_j^p-g_j)$ we have 
$A\simeq \frac{K[W_i,\ i\in I,W'_j\ j\in J]}{(P_i,\ i\in I,P'_j\ j\in J)}$.

Now we can apply the following general lemma 

\begin{lem} \label{lem5.1}
 Let $K$ be any field (no condition on the characteristic) and $K^{alg}$ an algebraic closure. 
 
 Let $n\geq 1$ and for $1\leq k\leq n$, $P_k\in K[W_k]$ be a non constant polynomial. Let $A$ be the 
$K$-algebra $\frac{K[W_k,1\leq k\leq n])}{(P_k)_{ 1\leq k\leq n}}=K[w_k,1\leq k\leq n]$ where $w_k$ is the canonical image of $W_k$. 

Let  $I\sqcup J$ be a partition of  $\{1,2,\cdots,n\}$ and  $B$ be the $K$-algebra $B:=\frac{K[W_k,k\in I]}{(P_k)_{ k\in I}}$. Let $u:K[W_k,k\in I]\to A$ be the $K$-homomorphism with $u(W_k)=w_k$ for $k\in I$ then $\ker u=\sum_{i\in I}P_i K[W_k,k\in I]$ and $u$ induces on one side an isomorphism between the two  $K$-algebras $B$ and $K[w_k,k\in I]\subset A$ and on the other side an isomorphism between the two  $K$-algebras $ \frac{B[W_k,k\in J]}{(P_k)_{k\in J}}$ and $A$.
\end{lem}

\begin{proof}

We have $\ker u:=\{P\in K[W_k,k\in I]\ |\ P=\sum_{1\leq k\leq n}Q_kP_k  \}$ where $Q_k\in K[W_k,1\leq k\leq n]$. Let $z_k\in K^{alg}$ with $P_k(z_k)=0$. Let $\sigma: K[W_k,1\leq k\leq n]\to K^{alg}[W_k,k\in I]$ such that $\sigma(a)=a$ for $a\in K$, $\sigma (W_k)=W_k$ for $k\in I$ and $\sigma(W_k)=z_k$ for $k\in J$, then 
\begin{eqnarray} 
P=\sigma(P)=\sum_{k\in I}\sigma (Q_k)P_k,\ {\rm and}\ \sigma(Q_k)\in K^{alg}[W_k,k\in I]. \label{star}
\end{eqnarray}
It follows that there is a finite field extension $L/K$ inside $K^{alg}$ with $\sigma(Q_k)\in L[W_k,k\in I]$. Let $\{e_0=1,e_1,\cdots,e_m\}$ a basis for $L/K$, then $L[W_k,k\in I]=\bigoplus_{0\leq s\leq m}K[W_k,k\in I]e_s$. It follows from \eqref{star} there is $R_k\in K[W_k,k\in I]$ with $P=\sum_{k\in I}R_kP_k$, thus $\ker u=\sum_{i\in I}P_i K[W_k,k\in I]$.

Let $\pi:K[W_k,k\in I]\to B$, be the canonical $K$-homomorphism and let $v:B\to A$ be the unique $K$-homomorphism with $u=v\circ \pi$. Then $v$ induces an isomorphism from $B$ to $K[w_k,k\in I]\subset A$.

So we have the following commutative diagram 

$$\begin{tikzcd}
K[W_k,k\in I] \arrow{d}{\pi}\arrow{r}{u}
& A  \\
B \ar[ur, "v" ,hook]
\end{tikzcd}$$

it extends in the following commutative diagram
 
$$\begin{tikzcd}
K[W_k,k\in I][W_k,k\in J] \arrow{d}{\tilde\pi}\arrow{r}{\tilde u}
& A  \\
B[W_k,k\in J] \ar[ur, "\tilde v"]
\end{tikzcd} $$

where $\tilde u (W_k)=w_k$, $\tilde\pi (W_k)=W_k$, $\tilde v (W_k)=w_k$ for $k\in J$.

We claim that $\ker \tilde v=\sum_{k\in J}P_kB[W_k,k\in J]$.

Let $Q\in \ker \tilde v$ and $\tilde Q\in K[W_k,k\in I][W_k,k\in J]$ such that $Q=\tilde \pi (\tilde Q)$. Then $\tilde u (\tilde Q)=\tilde v\tilde \pi (\tilde Q)=0$ and so $\tilde Q\in \sum_{1\leq k\leq n}K[W_t,1\leq t\leq n]P_k$ and 
$Q\in \tilde \pi(\sum_{1\leq k\leq n}P_kK[W_t,1\leq t\leq n])=\sum_{k\in J}P_kB[W_k,k\in J].$
 
\end{proof}

As $A\simeq \frac{K[W_i,\ i\in I,W'_j\ j\in J]}{(P_i,\ i\in I,P'_j\ j\in J)}$, it follows from Lemma \ref{lem5.1} that $A=\frac{K[w_i,\ i\in I][W'_j\ j\in J]}{(P'_j\ j\in J)} $  where \newline $K[w_i,\ i\in I]=\frac{K[W_i,\ i\in I]}{(P_i,\ i\in I)}$. Now with \eqref{39} we deduce that $A\simeq \frac{L[W'_j,j\in J]}{(P'_j)}\simeq L^{p^{n-r}}$. Moreover $A$ is a $K$-etale algebra since $L/K$ is separable.

ii) We show that the group $\Aut_K A$ is 
a semidirect product of the groups $\mathfrak S_{p^{n-r}}$ and $((\Z/p\Z)^r)^{p^{n-r}}$.

This follows from i) and the following lemma

\begin{lem} Let $K$ be a commutative field (no condition on the characteristic) and $L/K$ be a finite Galois extension  of group $G$. Let $t\geq 1$ and $A:=L^t$ and $\Aut_KA$ be the group of $K$-automorphisms of $A$. Let $\rho: \mathfrak S_t\to \Aut_KA$, 
 where $\rho(\sigma)(x_1,x_2,\cdots,x_t):=(x_{\sigma^{-1}(1 )},x_{\sigma^{-1}(2 )},\cdots,x_{\sigma^{-1}( t)})$ and $\varphi: G^t\to \Aut_KA$ such that
 
 \centerline{ $\varphi(g_1,g_2,\cdots,g_t)(x_1,x_2,\cdots,x_t):=(g_1(x_1),g_2(x_2),\cdots,g_t(x_t)),$}
 
 \noindent then $\rho$ and $\varphi$ are two injective homomorphisms of groups with
 
 \centerline{$\rho(\sigma)\varphi(g_1,g_2,\cdots,g_t)\rho(\sigma)^{-1}=\rho(g_{\sigma^{-1}(1 )},g_{\sigma^{-1}(2 )},\cdots,g_{\sigma^{-1}( t)})$}
 \noindent and $\Aut_KA$ is the 
 internal semidirect product of the groups $\rho(\mathfrak S_t)\simeq \mathfrak S_t$ and  $\varphi(G^t)\simeq G^t$.
\end{lem}

\begin{proof}

We can assume that $t\geq 2$ and we show the last assertion.

Let ${\mathfrak M}_i:=\{(x_1,\cdots,x_{i-1},0,x_{i+1},\cdots,x_t)\}$ with $x_j\in L$ for $j\neq i$, then ${\mathfrak M}_i$ is a maximal ideal of $A$  and $\frac{A}{{\mathfrak M}_i}\simeq L$. Then  $\{{\mathfrak M}_i,1\leq i\leq t\}$ is the set of maximal ideals  $\Spm(A)$ of $A$. 

Now if $\Phi\in \Aut_KA$, $\Phi$ induces a bijection of  $\Spm(A)$, so there is $\sigma\in \mathfrak S_t$ with $\Phi({\mathfrak M}_i)={\mathfrak M}_{\sigma^{-1}(i)}$ for $1\leq i\leq t$ hence $\rho(\sigma^{-1})\Phi({\mathfrak M}_i)={\mathfrak M}_i$  for $1\leq i\leq t$.

Let $\Psi:=\rho(\sigma^{-1})\Phi$, then for $1\leq i\leq t$ we have 
$\Psi (\cap_{j\neq i}{\mathfrak M}_j)=\cap_{j\neq i}{\mathfrak M}_j=(0,0,\cdots,L,0,\cdots,0)$ where only the $i$-th component is not zero. Thus 
$\Psi$ induces a $K$-automorphism $g_i$ of $L$. Thus we have $\Psi=\varphi(g_1,g_2,\cdots,g_t)$ and so $\Phi =\rho(\sigma)\varphi(g_1,g_2,\cdots,g_t)$. 

\end{proof}
iii) We show \eqref{F29}

 Let $v_i:=w_i+(f_i+f_i^p+...+f_i^{p^{n-2}}) $, then $w=\delta_{\underline{w}}(\underline{f})=\delta_{\underline{v}}(\underline{f})$ and  $v_i^p=v_i+f_i^{p^{n-1}}$. It follows that $w^{p^j}=
 \delta_{\underline{v}}(\underline{f}^{p^{j}})$ for $0\leq j\leq n-1$ and $w^{p^{n}}= \delta_{\underline{v}}(\underline{f}^{p^{n}})+\Delta_n(\underline{f}^{p^{n-1}}).$ 
 Since $\delta_{\underline{v}}(\underline{f}^{p^{j}})=\sum_{1\leq i\leq n}(-1)^{i-1}\Delta_{n-1}(\hat{f_i})^{p^j}v_i$, for $1\leq j\leq n$, we deduce from Cramer's formulas the announced formula for $w_i$ as a polynomial function of $w$.  Finally the previous formulas also give a non-trivial linear relation between the columns of the following determinant
 
$$\left|
\begin{array}{cccc}
\Delta_{n-1}(\underline{\hat{f}_1})&\cdots&\Delta_{n-1}(\underline{\hat{f}_n})&w\\
\Delta_{n-1}(\underline{\hat{f}_1})^p&\cdots&\Delta_{n-1}(\underline{\hat{f}_n})^p & w^p \\
\vdots&\cdots&\vdots&\vdots\\   
\Delta_{n-1}(\underline{\hat{f}_1})^{p^{n-1}}&\cdots&\Delta_{n-1}(\underline{\hat{f}_n})^{p^{n-1}}&w^{p^{n-1}} \\
\Delta_{n-1}(\underline{\hat{f}_1})^{p^{n}}&\cdots&\Delta_{n-1}(\underline{\hat{f}_n})^{p^{n}}&w^{p^{n}}-\Delta_{n}(\underline f^{p^{n-1}})
\end{array}
\right|
$$ 
 which is zero; hence we get \eqref{F29} where $Q(W):=\frac{\Delta_{n+1}(\Delta_{n-1}(\underline{\hat{f}_1}),\cdots,\Delta_{n-1}(\underline{\hat{f}_n}),W)}{\Delta_{n}(\Delta_{n-1}(\underline{\hat{f}_1}),\cdots,\Delta_{n-1}(\underline{\hat{f}_n}))  }-\Delta_n(\underline f)^{p^{n-1}}$ whose first term is the monic additive polynomial whose roots are the $\F_p$-space $\oplus_{1\leq i\leq n}\F_p\Delta_{n-1}(\underline{\hat{f}_i})$; then the equality $Q(W)=W^{p^n}+(\sum_{1\leq i\leq n-1}(-1)^{n-i}\Delta_n(\underline f)^{p^{n-1}-p^{i-1}-p^i}(\Delta[n-i](\underline f))^{p^{i-1}}W^{p^{n-i}}) +(-1)^n\Delta_n(\underline f)^{p^{n-1}-1}W-\Delta_n(\underline f)^{p^{n-1}},$ follows from Elkies (\cite{E} 4.28) and the proof of  Proposition \ref{Nprop2.4}.

\end{proof}

\begin{rem}
 i) Since $\Delta_r(f_{i},\ i\in I)\neq 0$,  Proposition \ref{prop5.2} applied to the $K$-algebra
 
 \noindent $L=\frac{K[W_{k},\ k\in I]}{(P_k)_{k\in I}}$ gives a generator of the extension $L/K$. 
 
 ii) One may consult (\cite{MT}) for an application in the case where $K=k((t))$ is a field of formal power series. 
\end{rem}

\begin{coro} \label{coro5.1} We keep the notations of the proposition.
 
Let $F:=\oplus_{1\leq i\leq n}\F_pf_i$, $Z:=\oplus_{1\leq i\leq n}\F_p\Delta_{n-1}(\underline{\hat{f}_i})$ be two $\F_p$-subspaces of $K$ associated to $\underline f$. Let $z\in Z$ and 
 $\sigma_z$ be the $K$-algebra automorphism of  $K[W]$ such that $\sigma_z(W):=W+z$; then 
 $\sigma_z$ induces a $K$-algebra automorphism of $A$ that we still denote by $\sigma_z$. The map $z\in Z\to \sigma_z \in \Aut_KA$ is  an injective group homomorphism  and its image is a subgroup $G$ of $\Aut_KA$ which is isomorphic to $(\Z/p\Z)^n$.
 Let $U:=(\frac{Z}{\Delta_n(\underline f)})^p\subset K$ be the $\F_p$-space of roots of the reversed polynomial  of $P_F(X):=\prod_{f\in F}(X-f)=\frac{\Delta_{n+1}(\underline{f},X)}{\Delta_{n}(\underline{f})}=X^{p^n}+\cdots+(-1)^n\Delta_{n}(\underline{f})^{p-1}X$  (cf. \eqref{F2} and Section 3.2).
 
 Let $z\in Z$, we can write $z=\Delta_{n}(\underline{f}) u^{1/p}$ with $u\in U$ and for $\underline\epsilon:=(\epsilon_1,\epsilon_2,...,\epsilon_n)\in \F_p^n-(0,0,\cdots,0)$ let $w_{\underline\epsilon }:= \sum_{1\leq i\leq n}\epsilon_iw_i\in A$ (resp. $f_{\underline\epsilon }:= \sum_{1\leq i\leq n}\epsilon_if_i\in K$) then $w_{\underline\epsilon }^p-w_{\underline\epsilon } =f_{\underline\epsilon }$ and  $K[w_{\underline\epsilon }]\subset A $ is isomorphic to the $K$-algebra $\frac{K[W_{\underline\epsilon}]}{(W_{\underline\epsilon }^p-W_{\underline\epsilon } -f_{\underline\epsilon }) }$ and is so a $K$-subalgebra of dimension $p$. Moreover 
 $\sigma_z(w_{\underline\epsilon })=w_{\underline\epsilon }+(-1)^{n-1}E(f_{\underline\epsilon },u)$ where $E: F\times U\to \F_p$ is the Elkies pairing (see Section 3.2.A.). 
 
 In particular when $r=n$ i.e. $A$ is a field and the group $G$ is the full group $\Aut_K A$, then
 the set $\{ K[w_{\underline\epsilon }]\ |\ \underline\epsilon \in \cal E \}$ where $\cal E$ is a set of representatives of $\mathbb{P}^{n-1}(\F_p)$, is equal to  the $\frac{p^n-1}{p-1}$, $p$-cyclic extensions of $K$ inside $A$. 
\end{coro}

\begin{proof}
i) We show the equality $\sigma_z(w_{\underline\epsilon })=w_{\underline\epsilon }+(-1)^{n-1}E(f_{\underline\epsilon },u)$.

We have $z:=\sum_{1\leq i\leq n}\alpha_i(-1)^{i-1}\Delta_{n-1}(\underline{\hat{f}_i})$ with $(\alpha_1,...,\alpha_n)\in \F_p^n$. Thus

 $\sigma_z(w_{\underline\epsilon })=\sum_{1\leq i\leq n}\epsilon_i(w_i+\frac{\Delta_n(\Delta_{n-1}(\underline{\hat{f}_1}),\cdots,(-1)^{i-2}\Delta_{n-1}(\underline{\hat{f}_{i-1}}),z,(-1)^{i}\Delta_{n-1}(\underline{\hat{f}_{i+1}}),\cdots, (-1)^{n-1}\Delta_{n-1}(\underline{\hat{f}_n}) )}{\Delta_n(\Delta_{n-1}(\underline{\hat{f}_1}),\cdots,(-1)^{j-1}\Delta_{n-1}(\underline{\hat{f}_{j}}),\cdots, (-1)^{n-1}\Delta_{n-1}(\underline{\hat{f}_n}) )}) =w_{\underline\epsilon } +\sum_{1\leq i\leq n}\epsilon_i\alpha_i=w_{\underline\epsilon }+(-1)^{n-1}E(f_{\underline\epsilon },u)$ where $E: (F,U)\to \F_p$ is the Elkies pairing (see 3.2.B and Proposition \ref{propAcc}). 
 
 ii) We show that $K[w_{\underline\epsilon }]\subset A $ is a $K$-subalgebra of dimension $p$.
 
 As the $w_i, 0\leq i\leq n$, are $\F_p$-linearly independant, after a $\F_p$-linear change of variables we can assume that 
 $\epsilon=(0,0,\cdots,n-1,1)$, then the results follows from 
 Lemma \ref{lem5.1}.
 
 The case $n=r$ in Corollary \ref{coro5.1} then follows from Galois theory.

\end{proof}

\bibliographystyle{alpha}

Jean Fresnel, 

\hfill Univ. Bordeaux, CNRS, Bordeaux INP, IMB, UMR 5251, 

\hfill F-33400, Talence, France 

\vskip 2mm \vskip 0pt

Michel Matignon,

\hfill Univ. Bordeaux, CNRS, Bordeaux INP, IMB, UMR 5251, 

\hfill F-33400, Talence, France 

\vskip 2mm \vskip 0pt
E-mail address: Jean.Fresnel$@$math.u-bordeaux.fr  

E-mail address: Michel.Matignon$@$math.u-bordeaux.fr

\end{document}